\documentclass[final]{siamltex}
\usepackage{mathrsfs}
\usepackage{graphicx}

\def\mathbb{\bf}
\def\leqslant{\le}
\def\geqslant{\ge}

\newtheorem{assumption}{Assumption}[section]
\newtheorem{remark}{Remark}[section]

\title{A Multi-dimensional Stochastic Singular Control Problem Via Dynkin Game and Dirichlet Form}
\author{Yipeng Yang \thanks{Department of Mathematics,
University of Missouri-Columbia, Columbia, Missouri, 65211
(yangyip@missouri.edu)}}

\date{}

\begin{document}

\maketitle

\markboth{MULTI-DIMENSIONAL SINGULAR CONTROL}{Y. YANG}

\begin{abstract} The traditional difficulty about
stochastic singular control is to characterize the regularities of
the value function and the optimal control policy. In this paper, a
multi-dimensional singular control problem is considered. We found
the optimal value function and the optimal control policy of this
problem via Dynkin game, whose solution is given by the saddle point
of the cost function. The existence and uniqueness of the solution
to this Dynkin game are proved through an associated variational
inequality problem involving Dirichlet form. As a consequence, the
properties of the value function of this Dynkin game implies the
smoothness of the value function of the stochastic singular control
problem. In this way, we are able to show the existence of a
classical solution to this multi-dimensional singular control
problem, which was traditionally solved in the sense of viscosity
solutions, and this enables the application of the verification
theorem to prove optimality.  \footnote{The idea of this paper was
proposed by Dr. Michael Taksar. Dr. Taksar passed away in February,
2012, however, his contributions should always be remembered.}
\end{abstract}

\section{Introduction and Problem Formulation}

The characterization of the regularities of value function and
optimal policy in stochastic singular control remains a big
challenge in stochastic control theory, especially the higher
dimensional case, see, e.g., \cite{Soner89}. The traditional
approach is to use the viscosity solution technique, see
\cite{Fleming06} \cite{Bayraktar12} \cite{Bassan02}, which usually
yields a less regular solution. Another approach to solve singular
control problems and characterize the regularity of value functions
is through variational inequalities and optimal stopping or Dynkin
game, see, e.g., Karatzas and Zamfirescu \cite{Karatzas05}, Guo and
Tomecek \cite{Guo08}. In \cite{Karatzas85}  Karatzas and Shreve
studied the connection between optimal stopping and singular
stochastic control of one dimensional Brownian motion, and showed
that the region of inaction in the control problem is the optimal
continuation region for the stopping problem. In
\cite{Baldursson97}, the authors established and exploited the
duality between the myopic investor's problem (optimal stopping) and
the social planning problem (stochastic singular control), where an
integral form and change of variable formula were also presented on
this connection. Ma \cite{ma92} dealt with a one dimensional
stochastic singular control problem where the drift term is assumed
to be linear and the diffusion term is assumed to be smooth, and he
showed that the value function is convex and $C^2$ and the
controlled process is a reflected diffusion over an interval. Guo
and Tomecek \cite{Guo09} solved a one dimensional singular control
problem via a switching problem \cite{Guo08}, and showed, using the
smooth fit property \cite{Pham07}, that under some conditions the
value function is continuously differentiable ($C^1$).

It is found that \cite{Fuku02} through the approach via game theory
and optimal stopping, it is possible to show the existence of a
smooth solution. The connection is the following: given a symmetric
Markov process on a locally compact separable metric space, it is
well known that the solution of an optimal stopping problem admits
its quasi continuous version of the solution to a variational
inequality problem involving Dirichlet form, e.g., see Nagai
\cite{Nagai78}. Zabczyk \cite{Zab84} extended this result to a
zero-sum game (Dynkin game). In the one dimensional case, the
integrated form of the value function of the Dynkin game was
identified to be the solution of an associated stochastic singular
control problem, e.g., see Taksar \cite{Taksar85}, Fukushima and
Taksar \cite{Fuku02} where a more general one dimensional diffusion
is assumed. As a result, the classical smooth solution ($C^2$) can
be obtained for this singular control problem.

This paper extends the work by Fukushima and Taksar \cite{Fuku02} to
multi-dimensional stochastic singular control problem. There are
many difficulties in this extension. In the one dimensional singular
control problem, each point in the space has a positive capacity
\cite{Fuku02}, hence the nonexistence of the proper exceptional set.
However, this is no longer the case in multi-dimensional singular
control problem. We overcome this difficulty using the absolute
continuity of the transition function of the underlying process
\cite{Fuku06}. Under some conditions, the optimal control policy of
the one dimensional case is proved to be the reflection of the
diffusion at two boundary points, but the form of the optimal
control policy and the conditions on the regularity of the value
function in multi-dimensional case are much more complicated. For
instance, in the two dimensional case, the boundary of the
continuation region can have various formats, e.g.,  bounded curves,
unbounded curves, singular points, disconnected curves, line
segments, etc. The difficulty in characterizing the continuation
region is due to the fact that its boundary is a free boundary, and
this paper investigates such issues.

In this paper, we are concerned with a multi-dimensional diffusion
on ${\mathbb{R}}^n$:
\begin{equation}\label{omodel}
d{\bf X}_t={\bf \mu}({\bf X}_t)dt+{\bf\sigma}({\bf X}_t)d{\bf B}_t,
\end{equation} where
\begin{displaymath}
{\bf X}_t=\left(\begin{array}{c}X_{1t}\\
\vdots\\
X_{nt}
\end{array}\right), \mu=\left(\begin{array}{c}\mu_1\\
\vdots\\
\mu_n
\end{array}\right), \sigma=\left(\begin{array}{ccc} \sigma_{11} &\  \cdots & \sigma_{1m}\\
\vdots &  & \vdots\ \\
\sigma_{n1} &\ \cdots & \sigma_{nm}\end{array}\right), {\bf
B}_t=\left(\begin{array}{c}B_{1t}\\
\vdots\\
B_{mt}\end{array}\right),
\end{displaymath} in which $\mu_i=\mu_i({\bf X}_t)$ and $\sigma_{i,j}=\sigma_{i,j}({\bf X}_t)$ ($1\leqslant i\leqslant n,1\leqslant j\leqslant
m$) are continuous functions of $X_{1t},X_{2t},...,X_{(n-1)t}$, and
${\bf B}_t$ is $m$-dimensional Brownian motion with $m\geqslant n$.
Thus we are given a system $(\Omega, \mathcal{F},\mathcal{F}_t, {\bf
X},\theta_t,P_{\bf x})$, where $(\Omega,\mathcal{F})$ is a
measurable space, ${\bf X}={\bf X}(\omega)$ is a mapping of $\Omega$
into $C({\mathbb{R}}^n)$, $\mathcal{F}_t=\sigma({\bf X}_s,s\leqslant
t)$, and $\theta_t$ is a shift operator in $\Omega$ such that ${\bf
X}_s(\theta_t\omega)={\bf X}_{s+t}(\omega)$. Here $P_{\bf x}$(${\bf
x}\in{\mathbb{R}}^n$) is a family of measures under which $\{{\bf
X}_t,t\geqslant 0\}$ is an $n$-dimensional diffusion with  initial
state ${\bf x}$. We assume that $\mu$ and $\sigma$ satisfy the usual
Lipschitz growth condition.

A control policy is defined as a pair
$(A_t^{(1)},A_t^{(2)})=\mathcal{S}$ of $\mathcal{F}_t$ adapted
processes which are right continuous and nondecreasing in $t$ and we
assume $A_0^{(1)},A_0^{(2)}$ are nonnegative. Denote $\mathbb{S}$
the set of all admissible policies, whose detailed definition will
be given in Section \ref{mdssc}.

Given a policy $\mathcal{S}=(A_t^{(1)},A_t^{(2)})\in\mathbb{S}$ we
define the following controlled process:
\begin{displaymath}\begin{array}{l}
dX_{1t}=\mu_1dt+\sigma_{11}dB_{1t}+\cdots+\sigma_{1m}dB_{mt},\\
\vdots  \quad\quad \vdots \quad\quad\quad \vdots\\
dX_{nt}=\mu_ndt+\sigma_{n1}dB_{1t}+\cdots+\sigma_{nm}dB_{mt}+dA_t^{(1)}-dA_t^{(2)},\\
{\bf X}_0={\bf x},
\end{array}
\end{displaymath}
with the cost function
\begin{eqnarray}\label{scost}
k_{\mathcal{S}}({\bf x})=E_{\bf x}\left(\int_0^\infty e^{-\alpha
t}h({\bf X_t})dt+\int_0^\infty e^{-\alpha t}\left(f_1({\bf
X}_t)dA_t^{(1)}+f_2({\bf X}_t)dA_t^{(2)}\right)\right),&&\\
f_1({\bf x}),f_2({\bf x})>0,\ \forall {\bf
x}\in{\mathbb{R}}^n.&&\nonumber
\end{eqnarray}
Here we assume that $A_t^{(1)}-A_t^{(2)}$ is the minimal
decomposition of a bounded variation process into a difference of
two increasing processes.

\begin{remark}A natural question is that why the control only applies on one
dimension. The difficulty arises in the step where the value
function of the zero-sum game is integrated (in one dimension) to
obtain the value function of the singular control problem. If the
control were applied to multi dimensions, no result so far is know
on the choice of the direction of integration. This represents a
traditional difficulty in multi-dimensional singular control
problem. Interested readers are referred to \cite{Soner89} for a
result on two dimensional singular control problem.
\end{remark}

There are two types of costs associated with the process ${\bf X_t}$
for each policy $\mathcal{S}$. The first one is the holding cost
$h({\bf X}_t)$ accumulated along time. The second one is the control
cost associated with the processes $(A_t^{(1)},A_t^{(2)})$, and this
cost increases only when $(A_t^{(1)},A_t^{(2)})$ increase.

One looks for a control policy $\mathcal{S}$ that minimizes
$k_{\mathcal{S}}({\bf x})$, i.e.,
\begin{equation}\label{mincostw}
W({\bf x})=\min_{\mathcal{S}\in\mathbb{S}} k_{\mathcal{S}}({\bf x}).
\end{equation}

As an application of this model, a decision maker observes the
expenses of a company under a multi-factor situation but only has
control over one factor, yet she still wants to minimize the total
expected cost. Analogously, by studying the associated maximization
problem, i.e., taking the negative of $\min$, this model can be used
to find the optimal investment policy where an investor observes the
prices of several assets in a portfolio and manages the portfolio by
adjusting one of them. Notice that every time there is a control
action, it yields a certain associated cost, e.g., the transaction
cost.

The rest of this paper is organized as follows: we first introduce
some preliminaries on Dirichlet form and a variational inequality
problem in Section \ref{DformDgame}. In Section \ref{DgameFBP} we
identify conditions for the value function as well as the optimal
policy of the associated Dynkin game. The integrated form of the
value function of this Dynkin game is shown in Section \ref{mdssc}
to be the value of a multi-dimensional singular control problem, and
the optimal control policy is also determined consequently. In the
appendix we shall correct an error found in the paper by Fukushima
and Taksar \cite{Fuku02}.

\section{Dirichlet Form and a Variational Inequality
Problem}\label{DformDgame}

Let $\mathbb{D}$ be a locally compact separable metric space,
${\mathbb m}$ be an everywhere dense positive Radon measure on
$\mathbb{D}$, and $L^2(\mathbb{D},m)$ denotes the $L^2$ space on
$\mathbb{D}$. We assume that the Dirichlet form
$(\mathcal{E},\mathscr{F})$ on $L^2(\mathbb{D},m)$ is regular in the
sense that $\mathscr{F}\cap C_0(\mathbb{D})$ is $\mathcal{E}_1$
dense in $\mathscr{F}$ and is uniformly dense in $C_0(\mathbb{D})$,
where the $\mathcal{E}_1$ norm is defined as follows:
\begin{displaymath}
\|u\|_{\mathcal{E}_1}=\left(\mathcal{E}(u,u)+\int_{\mathbb{D}}u({\bf
x})^2{\mathbb m}(d{\bf x})\right)^{1/2}.
\end{displaymath} Analogously we define
$\mathcal{E}_\alpha(u,v)$  as
$\mathcal{E}_\alpha(u,v)=\mathcal{E}(u,v)+\alpha(u,v)\ (\alpha>0)$,
where
\begin{displaymath}
(u,v)=\int_{\mathbb{D}}u({\bf x})v({\bf x}){\mathbb m}(d{\bf x}).
\end{displaymath}

For this Dirichlet form, there exists an associated Hunt process
${\bf M}=({\bf X}_t,P_{\bf x})$ on $\mathbb{D}$, see \cite{Fuku11},
such that

\begin{displaymath}
p_tf({\bf x}):=E_{\bf x}f({\bf X}_t),\quad {\bf x}\in\mathbb{D}
\end{displaymath}
is a version of $T_tf$ for all $f\in C_0(\mathbb{D})$, where $T_t$
is the $L^2$ semigroup associated with the Dirichlet form
$(\mathcal{E},\mathscr{F})$. Furthermore, the $L^2$-resolvent
$\{G_\alpha,\ \alpha>0\}$ associated with this Dirichlet form
satisfies
\begin{equation}\label{resol}
G_\alpha f\in\mathscr{F},\quad \mathcal{E}_\alpha(G_\alpha
f,u)=(f,u),\quad \forall f\in L^2(\mathbb{D};m),\ \forall
u\in\mathscr{F},
\end{equation} and the resolvent $\{R_\alpha,\ \alpha>0\}$ of the
Hunt process ${\bf M}$ given by
\begin{displaymath}
R_\alpha f({\bf x})=E_{\bf x}\left(\int_0^\infty e^{-\alpha t}f({\bf
X}_t)dt\right),\quad {\bf x}\in\mathbb{D},
\end{displaymath} is a quasi-continuous modification of $G_\alpha f$ for any Borel function $f\in
L^2(\mathbb{D};m)$.

For $\alpha>0$, a measurable function $f$ on $\mathbb{D}$ is called
$\alpha$-excessive if $f({\bf x})\geqslant 0$ and $e^{-\alpha t}p_t
f({\bf x})\uparrow  f({\bf x})$ as $t\downarrow  0$ for any ${\bf
x}\in\mathbb{D}$. A function $f\in\mathscr{F}$ is said to be an
$\alpha$-potential if $\mathcal{E}_\alpha(f,g)\geqslant 0$ for any
$g\in\mathscr{F}$ with $g\geqslant 0$. For any $\alpha$-potential
$f\in\mathscr{F}$, define $\hat{f}({\bf x})=\lim_{t\downarrow 0}p_t
f({\bf x})$, then $f=\hat{f}\ m$-a.e. and $\hat{f}$ is
$\alpha$-excessive (see Section 3 in \cite{Fuku06}). $\hat{f}$ is
called the $\alpha$-excessive regularization of $f$. Furthermore,
any $\alpha$-excessive function is finely continuous (see Theorem
A.2.7 in \cite{Fuku11}).

As  related literature, Nagai \cite{Nagai78} considered an optimal
stopping problem and showed that there exist a quasi continuous
function $w\in\mathscr{F}$ which solves the variational inequality
\begin{displaymath}
w\geqslant g,\quad \mathcal{E}_\alpha(w,u-w)\geqslant 0,\quad
\forall u\in\mathscr{F} {\rm \ with\ }u\geqslant g,
\end{displaymath} and a properly exceptional set $\mathbb{N}$ such
that for all ${\bf x}\in\mathbb{D}/\mathbb{N}$,
\begin{displaymath}
w({\bf x})=\sup_\sigma E_{\bf x}\left(e^{-\alpha\sigma}[g({\bf
X}_\sigma)]\right)=E_{\bf x}\left(e^{-\alpha\hat{\sigma}}[g({\bf
X}_{\hat{\sigma}})]\right),
\end{displaymath} where $g$ is a quasi continuous function in
$\mathscr{F}$ and
\begin{displaymath}\hat{\sigma}=\inf\{t\geqslant 0;w({\bf X}_t)=g({\bf
X}_t)\}.
\end{displaymath} Moreover, $w$ is the smallest $\alpha$-potential
dominating the function $g$ ${\mathbb m}$-a.e.

Zabczyk \cite{Zab84} then extended this result to the solution of
the zero-sum game (Dynkin game) by showing that there exist a quasi
continuous function $V({\bf x})\in\mathscr{F}$ which solves the
variational inequality
\begin{equation}\label{vineqH0}
g\leqslant V\leqslant h\ \ {\mathbb m}\ {\rm a.e.},\
\mathcal{E}_\alpha(V,u-V)\geqslant 0,\quad \forall u\in\mathscr{F},\
g\leqslant u\leqslant h\ \ {\mathbb m}\ {\rm a.e.},
\end{equation} and a properly exceptional set $\mathbb{N}$ such that for all ${\bf
x}\in\mathbb{D}/\mathbb{N}$,
\begin{equation}\label{saddleH0}
V({\bf x})=\sup_\sigma \inf_\tau J_{\bf x}(\tau,\sigma)=\inf_\tau
\sup_\sigma J_{\bf x}(\tau,\sigma)
\end{equation} for any stopping times $\tau$ and $\sigma$, where
\begin{equation}\label{JH0}
J_{\bf x}(\tau,\sigma)=E_{\bf
x}\left(e^{-\alpha(\tau\wedge\sigma)}\left(I_{\sigma\leqslant\tau}g({\bf
X}_\sigma)+I_{\tau<\sigma}h({\bf X}_\tau)\right)\right),
\end{equation} and $g\leqslant h$ $\mathbb m$-a.e. are
quasi-continuous functions in $\mathscr{F}$.

In these works, there always existed an exceptional set
$\mathbb{N}$. Fukushima and Menda \cite{Fuku06} showed that, if the
transition function of ${\bf M}$ satisfies an absolute continuity
condition, i.e.,
\begin{equation}\label{abscont} p_t({\bf x},\cdot)\ll {\mathbb m}(\cdot),
\end{equation} for all $t>0$ and ${\bf x}\in\mathbb{D}$, and
$g,h$ satisfy the following separability
condition:\\
\emph{There exist finite $\alpha$-excessive functions
$v_1,v_2\in\mathscr{F}$ such that, for all ${\bf x}\in\mathbb{D}$,
\begin{equation}\label{sepcond}
g({\bf x})\leqslant v_1({\bf x})-v_2({\bf x})\leqslant h({\bf x}),
\end{equation}}then Zabczyk's result still holds and there does not exist the
exceptional set $\mathbb{N}$. In what follows we shall introduce a
version of Theorem 2 in \cite{Fuku06}, where we used $-f_1,f_2$ in
places of $g,h$ respectively for the convenience of later use.

Let $f_1,f_2\in\mathscr{F}$ be finely continuous functions such that
for all ${\bf x}\in\mathbb{D}$
\begin{equation}\label{bdcond}
-f_1({\bf x})\leqslant f_2({\bf x}),\ |f_1({\bf
x})|\leqslant\phi({\bf x}),\ |f_2({\bf x})|\leqslant \psi({\bf x}),
\end{equation} where $\phi,\psi$ are some finite
$\alpha$-excessive functions, and $f_1,f_2$ are assumed to satisfy
the following separability condition
\begin{equation}\label{sepcondf}
-f_1({\bf x})\leqslant v_1({\bf x})-v_2({\bf x})\leqslant f_2({\bf
x}).
\end{equation}

We further define the set
\begin{equation}\label{dk}
K=\{u\in\mathscr{F}:\ -f_1\leqslant u\leqslant f_2,\ {\mathbb m}
{\rm -a.e.}\}.
\end{equation}
Considering the variational inequality problem
\begin{equation}\label{vineqH0f}
V\in K,\ \mathcal{E}_\alpha(V,u-V)\geqslant 0,\quad \forall u\in K,
\end{equation} we have:

\begin{theorem}\label{fuku2}
Assume conditions (\ref{abscont}), (\ref{bdcond})  and
(\ref{sepcondf}). There exists a finite finely continuous function
$V$ satisfying the variational inequality (\ref{vineqH0f}) and the
identity
\begin{displaymath}
V({\bf x})=\sup_\sigma\inf_\tau J_{\bf
x}(\tau,\sigma)=\inf_\tau\sup_\sigma J_{\bf x}(\tau,\sigma),\quad
\forall {\bf x}\in\mathbb{D},
\end{displaymath} where $\sigma,\tau$ range over all stopping times and
\begin{equation}\label{JH0f}
J_{\bf x}(\tau,\sigma)=E_{\bf
x}\left(e^{-\alpha(\tau\wedge\sigma)}\left(I_{\sigma\leqslant\tau}(-f_1({\bf
X}_\sigma))+I_{\tau<\sigma}f_2({\bf X}_\tau)\right)\right),
\end{equation}

Moreover, the pair $(\hat{\tau},\hat{\sigma})$ defined by
\begin{displaymath}
\hat{\tau}=\inf\{t>0:V({\bf X}_t)=f_2({\bf X}_t)\},\quad
\hat{\sigma}=\inf\{t>0:V({\bf X}_t)=-f_1({\bf X}_t)\},
\end{displaymath} is the saddle point of the game in the sense that
\begin{displaymath}
J_{\bf x}(\hat{\tau},\sigma)\leqslant J_{\bf
x}(\hat{\tau},\hat{\sigma})\leqslant J_{\bf
x}(\tau,\hat{\sigma}),\quad \forall {\bf x}\in\mathbb{D},
\end{displaymath} for all stopping times $\tau,\sigma$.
\end{theorem}

For a given function $H\in L^2(\mathbb{D};m)$ one looks for a
solution $V\in K$ to the following variational inequality problem
\begin{equation}\label{vi} \mathcal{E}_\alpha(V,u-V)\geqslant(H,u-V),\ \ \forall
u\in K.
\end{equation} Then we have the following
proposition:
\begin{proposition}\label{solVIeq} There exists a unique finite finely continuous function $V\in K$
which solves (\ref{vi}).
\end{proposition}
\begin{proof} The proof is essentially identical to the proof of
Proposition 2.1 in \cite{Fuku02} and is omitted here.

\end{proof}

We assume further the following separability condition:
\begin{assumption} There exist finite
$\alpha$-excessive functions $v_1,v_2\in\mathscr{F}$ such that, for
all ${\bf x}\in\mathbb{D}$,
\begin{equation}\label{sepcondHneq0}
-f_1({\bf x})-G_\alpha H({\bf x})\leqslant v_1({\bf x})-v_2({\bf
x})\leqslant f_2({\bf x})-G_\alpha H({\bf x}),\end{equation}
\end{assumption}
then the following result holds:

\begin{theorem}\label{DgameHneq0}
For any function $H\in L^2({\mathbb D};{\mathbb m})$ ) and any
$f_1,f_2\in\mathscr{F}$ such that $f_1({\bf x})+G_\alpha H({\bf x})$
and $f_2({\bf x})-G_\alpha H({\bf x})$ are finely continuous and
bounded by some finite $\alpha$-excessive functions, respectively.
Assuming (\ref{abscont})(\ref{sepcondHneq0}), we put
\begin{eqnarray}\label{Jcost}
J_{\bf x}(\tau,\sigma)&=&E_{\bf x}\left(\int_0^{\tau\wedge\sigma}
e^{-\alpha t}H({\bf X}_t)dt\right)\\
&&+E_{\bf x}\left( e^{-\alpha
(\tau\wedge\sigma)}\left(-I_{\sigma\leqslant\tau}f_1({\bf
X}_\sigma)+I_{\tau<\sigma}f_2({\bf X}_\tau)\right)\right)\nonumber
\end{eqnarray} for any stopping times $\tau,\sigma$. Then the
solution of (\ref{vi})  admits a finite finely continuous value
function of the game
\begin{equation}
V({\bf x})=\inf_\tau\sup_\sigma J_{\bf
x}(\tau,\sigma)=\sup_\sigma\inf_\tau J_{\bf x}(\tau,\sigma),\quad
\forall {\bf x}\in\mathbb{D}.
\end{equation}

Furthermore if we let \begin{displaymath} E_1=\{{\bf x}\in
{\mathbb{D}}: V({\bf x})=-f_1({\bf x})\}, \quad E_2=\{{\bf x}\in
{\mathbb{D}}: V({\bf x})=f_2({\bf x})\},
\end{displaymath} then the hitting times $\hat{\tau}=\tau_{E_2}$,
$\hat{\sigma}=\tau_{E_1}$ is the saddle point of the game
\begin{equation}\label{saddleHn0}
J_{\bf x}(\hat{\tau},\sigma)\leqslant J_{\bf
x}(\hat{\tau},\hat{\sigma})\leqslant J_{\bf x}(\tau,\hat{\sigma})
\end{equation} for any ${\bf x}\in \mathbb{D}$ and any stopping times $\tau,\sigma$. In
particular,
\begin{equation}
V({\bf x})=J_{{\bf x}}(\hat{\tau},\hat{\sigma}),\quad \forall {\bf
x}\in \mathbb{D}.
\end{equation}
\end{theorem}
$E_1$ is the set of points where $V=-f_1$ and $E_2$ is the set of
points where $V=f_2$. So $\hat{\tau}$ and $\hat{\sigma}$ in Theorem
\ref{DgameHneq0} can be defined in the same way as in Theorem
\ref{fuku2}. The proof of Theorem \ref{DgameHneq0} is identical to
Theorem 2.1 in \cite{Fuku02}.

\section{The Dynkin Game and Its Value Function}\label{DgameFBP}
Two players $P_1$ and $P_2$ observe a multi-dimensional underlying
process ${\bf X}_t$ in (\ref{omodel}) with accumulated income,
discounted at present time, equalling $\int_0^\sigma e^{-\alpha
t}H({\bf X}_t)dt$ for any stopping time $\sigma$. If $P_1$ stops the
game at time $\sigma$, he pays $P_2$ the amount of the accumulated
income plus the amount $f_2({\bf X}_\sigma)$, which after been
discounted equals $e^{-\alpha \sigma}f_2({\bf X}_\sigma)$. If the
process is stopped by $P_2$ at time $\sigma$, he receives from $P_1$
the accumulated income less the amount $f_1({\bf X}_\sigma)$, which
after been discounted equals $e^{-\alpha \sigma}f_1({\bf
X}_\sigma)$. $P_1$ tries to minimize his payment while $P_2$ tries
to maximize his income. Let $\tau,\sigma$ be two stopping times, the
value of this game is thus given by
\begin{equation}\label{vgame}
V({\bf x})=\inf_\tau\sup_\sigma J_{{\bf x}}(\tau,\sigma),\quad
\forall {\bf x}\in {\mathbb{R}}^n,
\end{equation} where $J_{\bf x}$ is given by (\ref{Jcost}) on ${\mathbb{R}}^n$.

For the diffusion (\ref{omodel}), define its infinitesimal generator
$\mathcal{L}$ as
\begin{equation}\label{infgen}
\mathcal{L}:=\sum_{i=1}^n\mu_i\frac{\partial}{\partial
x_i}+\sum_{i,j=1}^n A_{ij}\frac{\partial^2}{\partial x_i\partial
x_j},
\end{equation} where $[A_{ij}]={\bf A}=\frac{1}{2}\sigma
\sigma^T$. We assume that ${\bf A}$ is non-degenerate.

Define the measure ${\mathbb m}(d{\bf x})=\rho({\bf x})d{\bf x}$,
where $\rho({\bf x})$ satisfies the following condition:
\begin{equation}\label{rhocond}
{\mathbb A}\nabla\rho=\rho\cdot(\mu-{\bf b}),
\end{equation} where $b_i=\nabla\cdot{\bf A}_i,i=1,2,...,n$ in ${\bf
b}$. (Notice that when $\mu$ and ${\mathbb A}$ are constants,
$\rho({\bf x})$ reduces to $\rho({\bf x})=\exp(({\mathbb
A}^{-1}\mu)\cdot {\bf x})$.) It can be seen that the absolute
continuity condition (\ref{abscont}) is satisfied.

\begin{remark}
We are unable to solve the case with a general multidimensional
diffusion. Even in the case of one dimensional diffusion, conditions
on $\mu$ and $\sigma$ should be made (see Appendix).
\end{remark}

For the generator $\mathcal{L}$, its associated Dirichlet form
$(\mathcal{E},\mathscr{F})$ densely embedded in
$L^2({\mathbb{R}^n};{\mathbb m})$ is then given by
\begin{equation}\label{NDform}
\mathcal{E}(u,v)=\int_{\mathbb{R}^n}\nabla u({\bf x})\cdot{\bf
A}\nabla v({\bf x}){\mathbb m}(d{\bf x}), \quad u,v\in\mathscr{F},
\end{equation}
where
\begin{displaymath}
\mathscr{F}=\{u\in L^2({\mathbb{R}^n};{\mathbb m}):\ u {\rm\ is\
 continuous},\ \int_{{\mathbb{R}^n}}\nabla u({\bf
x})^T\nabla u({\bf x}){\mathbb m}(d{\bf x})<\infty\}.
\end{displaymath}

For given functions $H,f_1,f_2$ satisfying the conditions of Theorem
\ref{DgameHneq0}, and noticing that ${\bf X}_t$ is a non-degenerate
Ito diffusion,  we can conclude that $V({\bf x})$ in
Eq.(\ref{vgame}) is finite and continuous, and it solves (\ref{vi}).
Furthermore if we let \begin{equation}\label{e12R} E_1=\{{\bf x}\in
{\mathbb{R}}^n: V({\bf x})=-f_1({\bf x})\}, \quad E_2=\{{\bf x}\in
{\mathbb{R}}^n: V({\bf x})=f_2({\bf x})\},
\end{equation} then the hitting times $\hat{\tau}=\tau_{E_2}$,
$\hat{\sigma}=\tau_{E_1}$ is the saddle point of the game
\begin{equation}\label{saddle_Cont}
J_{\bf x}(\hat{\tau},\sigma)\leqslant J_{\bf
x}(\hat{\tau},\hat{\sigma})=V({\bf x})\leqslant J_{\bf
x}(\tau,\hat{\sigma})
\end{equation} for any ${\bf x}\in {\mathbb{R}}^n$ and any stopping times
$\tau,\sigma$.

In the next section we shall give conditions on $H,f_1,f_2$ and
characterize the regularities of $V({\bf x})$ and the form of the
optimal control policy.

\subsection{Optimal Stopping Regions} In the one dimensional case,
if the functions are defined over a bounded interval, a lot of
properties are automatically satisfied \cite{Fuku02}. But in
multi-dimensional case, this is much harder.

It is obvious that the conditions on $H,f_1,f_2$ are critical on the
form of optimal control policy. For example, if $H\equiv 0$ and
$-f_1({\bf x})<0<f_2({\bf x}),\ \forall {\bf x}$, then no party
would ever stop the game and there is no optimal control.

\begin{assumption}\label{assumpHf}
$f_1,f_2\in\mathscr{F}$ are smooth functions, $-M<-f_1({\bf
x})<0<f_2({\bf x})<M,\ \forall {\bf x}\in\mathbb{R}^n$ where $M$ is
a constant, and $H\in L^2(\mathbb{R}^n;\mathbb{m})$ is everywhere
continuous, and the separability condition (\ref{sepcondHneq0})
holds. $H(\bar{\bf x},x_n)$ is strictly increasing in $x_n$,
$f_1(\bar{\bf x},x_n)$ is nondecreasing in $x_n$, $f_2(\bar{\bf
x},x_n)$ is nonincreasing in $x_n$. Further more,
$(\alpha-\mathcal{L})f_1(\bar{\bf x},x_n)+H(\bar{\bf x},x_n)$ is
strictly increasing in $x_n$ and $(\alpha-\mathcal{L})f_2(\bar{\bf
x},x_n)-H(\bar{\bf x},x_n)$ is strictly decreasing in $x_n$. The
(hyper)curves $a(\bar{\bf x})$, $b(\bar{\bf x})$ such that
\begin{eqnarray*}
(\alpha-\mathcal{L})f_1(\bar{\bf x},a(\bar{\bf x}))+H(\bar{\bf
x},a(\bar{\bf x}))&=&0,\\
(\alpha-\mathcal{L})f_2(\bar{\bf x},b(\bar{\bf x}))-H(\bar{\bf
x},b(\bar{\bf x}))&=&0,
\end{eqnarray*} with $a(\bar{\bf x})<b(\bar{\bf
x})$, $\forall \bar{\bf x}\in{\mathbb{R}}^{n-1}$, are assumed to be
bounded and uniformly Lipschitz continuous.
\end{assumption}

 Then it is easy to see that
\begin{proposition}\label{f1f2Hsign} Assume Assumption
\ref{assumpHf}. For any $(\bar{\bf x},x_n)$ with $x_n<a(\bar{\bf
x})$,
\begin{displaymath}
(\alpha-\mathcal{L})f_1(\bar{\bf x},x_n)+H(\bar{\bf x},x_n)<0,
\end{displaymath} and for any $(\bar{\bf x},x_n)$ with $x_n>a(\bar{\bf
x})$,
\begin{displaymath}
(\alpha-\mathcal{L})f_1(\bar{\bf x},x_n)+H(\bar{\bf x},x_n)>0.
\end{displaymath}
Similarly, for any $(\bar{\bf x},x_n)$ with $x_n<b(\bar{\bf x})$,
\begin{displaymath}
(\alpha-\mathcal{L})f_2(\bar{\bf x},x_n)-H(\bar{\bf x},x_n)>0,
\end{displaymath} and for any $(\bar{\bf x},x_n)$ with $x_n>b(\bar{\bf
x})$,
\begin{displaymath}
(\alpha-\mathcal{L})f_2(\bar{\bf x},x_n)-H(\bar{\bf x},x_n)<0.
\end{displaymath}
\end{proposition}

Define the set
\begin{equation}\label{cregion}
E=\{{\bf x}\in{\mathbb{R}}^n:-f_1({\bf x})<V({\bf x})<f_2({\bf
x})\}.
\end{equation}Since $P_2$ would stop the game once
$V({\bf x})\leqslant -f_1({\bf x})$ and the instant payoff is
$-f_1({\bf x})$, while $P_1$ would stop the game once $V({\bf
x})\geqslant f_2({\bf x})$ and the instant payoff is $f_2({\bf x})$,
we could write ${\mathbb{R}}^n$ as a partition:
\begin{displaymath}
{\mathbb{R}}^n=E_1\cup E\cup E_2,
\end{displaymath} where $E_1,E_2$ were given in (\ref{e12R}).

\begin{proposition}\label{Jandf} Assume Assumption
\ref{assumpHf}.
 For each ${\bf x}\in E_1$,
\begin{displaymath}
(\alpha-\mathcal{L})f_1({\bf x})+H({\bf x})\leqslant0,
\end{displaymath}
and for each ${\bf x}\in E_2$,
\begin{displaymath}(\alpha-\mathcal{L})f_2({\bf
x})-H({\bf x})\leqslant 0.
\end{displaymath}
\end{proposition}

\begin{proof}
We only give proof to the first half. We know at the point ${\bf
x}\in E_1$ it must be true that $V({\bf x})\leqslant-f_1({\bf x})$,
and it is optimal for $P_2$ to stop the game immediately. Suppose
\begin{displaymath}
(\alpha-\mathcal{L})f_1({\bf x})+H({\bf x})>0,\ {\bf x}\in E_1,
\end{displaymath} then by the smoothness of $f_1$ and the continuity
of $H$, we can find a small ball $B_r({\bf x})$ containing the point
${\bf x}$, such that for each ${\bf y}\in B_r({\bf x})$,
\begin{displaymath}
(\alpha-\mathcal{L})f_1({\bf y})+H({\bf y})>0.
\end{displaymath}
Consider a policy for $P_2$ to stop the game at the first exit time
of $B_r({\bf x})$, denoted $\tau_{B_r}$. Then by Dynkin's formula,
the payoff would be
\begin{eqnarray*}
J_{{\bf x}}&=&E_{{\bf x}}\int_0^{\tau_{B_r}}e^{-\alpha t}H({\bf
X}_t)dt+E_{{\bf
x}}(e^{-\alpha \tau_{B_r}}(-f_1({\bf X}_{\tau_{B_r}})))\\
&=&E_{{\bf x}}\int_0^{\tau_{B_r}}e^{-\alpha t}H({\bf
X}_t)dt-f_1({\bf x})+E_{{\bf x}}\int_0^{\tau_{B_r}}e^{-\alpha
t}(\alpha-\mathcal{L})f_1({\bf X}_t)dt\\
&=&-f_1({\bf x})+E_{{\bf x}}\int_0^{\tau_{B_r}}e^{-\alpha
t}[(\alpha-\mathcal{L})f_1({\bf
X}_t)+H({\bf X}_t)]dt\\
&>&-f_1({\bf x}).
\end{eqnarray*} This is a contradiction since $P_2$ tries to
maximize his payoff but we assumed that the optimal policy at ${\bf
x}$ was to stop the game immediately.
\end{proof}

\begin{corollary} Assume Assumption
\ref{assumpHf}. If ${\bf x}=(\bar{\bf x},x_n)\in E_1$, then for any
point $(\bar{\bf x},y)$ with $y< x_n$,
\begin{displaymath}
(\alpha-\mathcal{L})f_1(\bar{\bf x},y)+H(\bar{\bf x},y)<0.
\end{displaymath}
If ${\bf x}=(\bar{\bf x},x_n)\in E_2$, then for any point $(\bar{\bf
x},y)$ with $y> x_n$,
\begin{displaymath}
(\alpha-\mathcal{L})f_2(\bar{\bf x},y)-H(\bar{\bf x},y)<0.
\end{displaymath} Furthermore,
$E_1\subseteq{\mathbb R}^{n-1}\times(-\infty,a]$ and
$E_2\subseteq{\mathbb R}^{n-1}\times[b,\infty)$.
\end{corollary}
\begin{proof}
This can be easily seen from Proposition \ref{Jandf} and the
conditions on $f_1,f_2, H$ given in Assumption
\ref{assumpHf}.\end{proof}

Further, noticing the conditions on the curves $a({\bar{\bf x}})$
and $b({\bar{\bf x}})$, $\bar{\bf x}\in{\mathbb R}^{n-1}$, we have
the following:
\begin{corollary} Assume Assumption
\ref{assumpHf}. $E\supseteq {\mathbb R}^{n-1}\times(a,b)$ and hence
$E$ is not empty. Furthermore, the value of this game $V$ is bounded
by $M$, where $M$ is given in Assumption \ref{assumpHf}.
\end{corollary}

Take any point ${\bf x}=({\bar{\bf x}},x_n)\in E_1$, and denote
$\sigma_a$ the hitting time to the curve $a(\cdot)$. Notice that the
diffusion (\ref{omodel}) is a  conservative process by the given
conditions, and also by noticing the conditions given on $a(\cdot)$,
it can be concluded that $E_{\bf x}(e^{-\alpha \sigma_a})$ goes to
zero as $x_n$ goes to $-\infty$. Similarly $E_{\bf x}(e^{-\alpha
\sigma_b})$, ${\bf x}=({\bar{\bf x}},x_n)\in E_2$, goes to zero as
$x_n$ goes to $\infty$.

\begin{assumption}\label{assumAB}
There exist functions $A(\bar{\bf x}), B(\bar{\bf x})$, $\bar{\bf
x}\in{\mathbb{R}}^{n-1}$ that are uniformly bounded and such that
for any ${\bf x}\in{\mathbb R}^{n-1}\times(-\infty,A]$,
\begin{displaymath}
E_{\bf x}\left(\int_0^{\sigma_a} e^{-\alpha t}H({\bf
X}_t)dt+e^{-\alpha\sigma_a}M\right)<-f_1({\bf x}),
\end{displaymath} and for any ${\bf x}\in{\mathbb
R}^{n-1}\times[B,\infty)$,
\begin{displaymath}
E_{\bf x}\left(\int_0^{\sigma_b} e^{-\alpha t}H({\bf
X}_t)dt-e^{-\alpha\sigma_b}M\right)>f_2({\bf x}).
\end{displaymath}
\end{assumption}

\begin{proposition} Assume Assumptions \ref{assumpHf} and
\ref{assumAB}, then $A(\bar{\bf x})<a(\bar{\bf x})$ and $B(\bar{\bf
x})>b(\bar{\bf x})$, $\forall\bar{\bf x}\in{\mathbb{R}}^{n-1}$.
Furthermore, on ${\mathbb R}^{n-1}\times(-\infty,A]$ player $P_2$
would stop the game immediately and $V=-f_1$, and on ${\mathbb
R}^{n-1}\times[B,\infty)$ player $P_1$ would stop the game
immediately and $V=f_2$.
\end{proposition}
\begin{proof}
Suppose there is a point ${\bf x}=(\bar{\bf x},x_n)$ with
$a(\bar{\bf x})\leq x_n\leq A(\bar{\bf x})$. Then by Dynkin's
formula,
\begin{eqnarray*}
&&E_{\bf x}\left(\int_0^{\sigma_a\wedge T} e^{-\alpha t}H({\bf
X}_t)dt+e^{-\alpha(\sigma_a\wedge T)}M\right)\\
&& > E_{\bf x}\left(\int_0^{\sigma_a\wedge T} e^{-\alpha t}H({\bf
X}_t)dt+e^{-\alpha(\sigma_a\wedge T)}(-f_1({\bf X}_{\sigma_a\wedge
T}))\right)\\
&&=-f_1({\bf x})+E_{\bf x}\left(\int_0^{\sigma_a\wedge T}e^{-\alpha
t}((\alpha-\mathcal{L})f_1({\bf X}_t)+H({\bf
X}_t))dt\right)>-f_1({\bf x})
\end{eqnarray*} by Proposition \ref{f1f2Hsign}. Taking $T\to\infty$
we get a contradiction.

Now suppose ${\bf x}\in{\mathbb R}^{n-1}\times(-\infty,A]$. For any
stopping time $\sigma$ for player $P_2$  the payoff  will be
\begin{eqnarray*}
&&E_{\bf x}\int_0^{\sigma}e^{-\alpha t}H({\bf X}_t)dt+E_{\bf
x}[e^{-\alpha \sigma}(-f_1({\bf X}_{\sigma}))]\\
&&=\left(E_{\bf x}\int_0^{\sigma}e^{-\alpha t}H({\bf X}_t)dt+E_{\bf
x}[e^{-\alpha \sigma}(-f_1({\bf X}_{\sigma}))]\right)P_{\bf
x}(\sigma\leq \sigma_a)\\
&&+\left(E_{\bf x}\int_0^{\sigma}e^{-\alpha t}H({\bf X}_t)dt+E_{\bf
x}[e^{-\alpha \sigma}(-f_1({\bf X}_{\sigma}))]\right)P_{\bf
x}(\sigma>\sigma_a).
\end{eqnarray*}
When $\sigma\leq\sigma_a$, the following quantity
\begin{eqnarray*}
&&=E_{\bf x}\int_0^{\sigma\wedge T\wedge\sigma_a}e^{-\alpha t}H({\bf
X}_t)dt+E_{\bf
x}[e^{-\alpha (\sigma\wedge T\wedge \sigma_a)}(-f_1({\bf X}_{\sigma\wedge T\wedge\sigma_a}))]\\
&&=-f_1({\bf x})+E_{\bf x}\int_0^{\sigma\wedge
T\wedge\sigma_a}e^{-\alpha t}((\alpha-\mathcal{L})f_1({\bf
X}_t)+H({\bf X}_t))dt
\end{eqnarray*} is less than $-f_1({\bf x})$, $\forall T>0$ by Proposition \ref{f1f2Hsign}.

\noindent When $\sigma>\sigma_a$,
\begin{eqnarray*}
 &&E_{\bf x}\int_0^{\sigma}e^{-\alpha t}H({\bf X}_t)dt+E_{\bf x}[e^{-\alpha \sigma}(-f_1({\bf
X}_{\sigma}))]\\
&&\leq E_{\bf x}\left(\int_0^{\sigma_a}e^{-\alpha t}H({\bf
X}_t)dt+e^{-\alpha\sigma_a}M\right)
\end{eqnarray*} because $M$ is the bound of the payoff of each
player. Hence by Assumption \ref{assumAB},
\begin{displaymath}
E_{\bf x}\int_0^{\sigma}e^{-\alpha t}H({\bf X}_t)dt+E_{\bf
x}[e^{-\alpha \sigma}(-f_1({\bf X}_{\sigma}))]<-f_1({\bf x}).
\end{displaymath}
As a summary, if ${\bf x}\in{\mathbb R}^{n-1}\times(-\infty,A]$,
then for any stopping policy of $P_2$, the expected payoff is less
than $-f_1({\bf x})$, and the optimal strategy is to stop the game
immediately. The other half of this proposition can be proved in a
similar way.
\end{proof}

By the properties of $A(\cdot)$ (or $B(\cdot)$), we can choose a
bounded and continuous curve below $A(\cdot)$ (or a bounded and
 continuous curve above $B(\cdot)$) which also
has the properties as given in Assumption \ref{assumAB}. Without
loss of generality, we assume  $A(\cdot)$ and $B(\cdot)$ are bounded
and continuous.

Now it is easy to see that ${\mathbb
R}^{n-1}\times(-\infty,A]\subseteq E_1$ and ${\mathbb
R}^{n-1}\times[B,\infty)\subseteq E_2$. Since the functions
$V,f_1,f_2$ are all continuous, the boundary of $E$ consists of
continuous curves.

Let $\tilde{E}_1$ be the largest connected region in $E_1$
containing the set ${\mathbb R}^{n-1}\times(-\infty,A]$, and
$\tilde{E}_2$ be the largest connected region in $E_2$ containing
the set ${\mathbb R}^{n-1}\times[B,-\infty)$, then obviously $V({\bf
x})=-f_1({\bf x}),\forall {\bf x}\in\tilde{E}_1$, and $V({\bf
x})=f_2({\bf x}),\forall {\bf x}\in\tilde{E}_2$. Furthermore
$\tilde{E}_1$ has a continuous boundary curve $\tilde{a}$ that is
bounded by $A$ and $a$, $\tilde{E}_2$ has a continuous boundary
curve $\tilde{b}$ that is bounded by $b$ and $B$.


\begin{proposition}
$\tilde{E}_1,\tilde{E}_2$ are simply connected regions.
\end{proposition}
\begin{proof}
It suffices to prove that $\tilde{E}_1$ is simply connected. Suppose
there is a point ${\bf x}_0$ such that $V({\bf x}_0)>-f_1({\bf
x}_0)$, and the boundary $\partial D_{\bf x_0}$ of the largest
connected region containing ${\bf x}_0$ as well as the points ${\bf
x}$ with $V({\bf x})>-f_1({\bf x})$ belongs to $\tilde{E}_1$, i.e.,
$\partial D_{\bf x_0}\subset\tilde{E}_1$. Then for any ${\bf x}\in
\partial D_{\bf x_0}$, $V({\bf x})=-f_1({\bf x})$, and for any ${\bf x}\in D_{\bf
x_0}$, $(\alpha-\mathcal{L})f_1({\bf x})+H({\bf x})<0$. Consider any
stopping strategy for player $P_2$ with the reward
\begin{eqnarray}
&&E_{\bf x_0}\left(\int_0^\sigma e^{-\alpha t}H({\bf
X}_t)dt+e^{-\alpha\sigma}(-f_1({\bf X}_\sigma))\right)\nonumber\\
&&=-f_1({\bf x}_0)+E_{\bf x_0}\int_0^\sigma e^{-\alpha
t}((\alpha-\mathcal{L})f_1({\bf X}_t)+H({\bf X}_t))dt\nonumber\\
&&<-f_1({\bf x}_0),
\end{eqnarray} since $\sigma\leqslant \tau_{\partial D_{\bf x_0}}$ a.s.,
where $\tau_{\partial D_{\bf x_0}}$ is the first hitting time to
$\partial D_{\bf x_0}$, and on $\partial D_{\bf x_0}$ it is known
that the game should be stopped by $P_2$ with the payoff $-f_1$. But
this is a contradiction since we have assumed that $V({\bf
x}_0)>-f_1({\bf x}_0)$. Therefore there is no \emph{hole} in
$\tilde{E}_1$.
\end{proof}

\begin{proposition}
If there is any point $(\bar{\bf x},x_n)$ with $x_n<a(\bar{\bf x})$
such that $V(\bar{\bf x},x_n)>-f_1(\bar{\bf x},x_n)$, then  the
connected region containing this point with $V>-f_1$ is connected to
the region ${\mathbb{R}}^{n-1}\times[a,b]$. A similar result holds
on the curve $b$.
\end{proposition}
\begin{proof}
Suppose not, then the boundary of the connected region containing
$(\bar{\bf x},x_n)$ with $V>-f_1$ is contained in ${\mathbb
R}^{n-1}\times(-\infty,a)$. For any stopping strategy for player
$P_2$, the process is stopped before it hits the curve $a$. The
expected payoff is less than or equal to $-f_1(\bar{\bf x},x_n)$
because in this region $(\alpha-\mathcal{L})f_1+H<0$, and this
contradicts the assumption $V(\bar{\bf x},x_n)>-f_1(\bar{\bf
x},x_n)$.
\end{proof}

Now it is clear that the region $E$ is connected. Recall that
$\hat{\sigma}=\tau_{E_1},\hat{\tau}=\tau_{E_2}$ are the first
hitting times to the sets $E_1,E_2$ respectively, and they are
finite a.s., we can rewrite $V$ as
\begin{displaymath}
V({\bf x})=E_{\bf
x}\left(\int_0^{\hat{\tau}\wedge\hat{\sigma}}e^{-\alpha t}H({\bf
X}_t)dt+I_{\hat{\sigma}<\hat{\tau}}e^{-\alpha\hat{\sigma}}(-f_1({\bf
X}_{\hat{\sigma}}))+I_{\hat{\tau}<\hat{\sigma}}e^{-\alpha\hat{\tau}}f_2({\bf
X}_{\hat{\tau}})\right).
\end{displaymath}

But at this point we still can not tell that $E$ is simply
connected. Let ${\bf U}$ be any connected region such that ${\mathbb
R}^{n-1}\times[a,b]\subset {\bf U}$ and ${\bf U}\subset {\mathbb
R}^{n-1}\times[A,B]$. Define $\tau_{\bf U}$ the first exit time of
this region, then obviously $\tau_{\bf U}$ is finite a.s. Define the
function
\begin{equation}\label{fdef}
F_{\bf U}({\bf x})=E_{\bf x}\left(\int_0^{\tau_{\bf U}}e^{-\alpha
t}H({\bf X}_t)dt+e^{-\alpha\tau_{\bf U}}R({\bf X}_{\tau_{\bf
U}})\right),\ \ \forall {\bf x}\in{\bf U},
\end{equation} where $R({\bf X}_{\tau_{\bf U}})=-f_1({\bf X}_{\tau_{\bf
U}})$ if ${\bf X}_{\tau_{\bf U}}\in{\mathbb R}^{n-1}\times[A,a]$,
and $R({\bf X}_{\tau_{\bf U}})=f_2({\bf X}_{\tau_{\bf U}})$ if ${\bf
X}_{\tau_{\bf U}}\in{\mathbb R}^{n-1}\times[b,B]$.

We put an assumption jointly on the process ${\bf X}_t$ and the
functions $f_1,f_2,h$. Consider any point $(\bar{\bf
x}_0,x_n)\in\bar{\bf U}$ with $x_n<a(\bar{\bf x}_0)$. Define the
cone \begin{displaymath}C_{(\bar{\bf x}_0,x_n)}=\left\{(\bar{\bf
x},y):x_n\leqslant y\leqslant a(\bar{\bf
x}),\frac{y-x_n}{\|<\bar{\bf x}-\bar{\bf x}_0,y-x_n>\|_2}\geqslant
\zeta\right\},\end{displaymath} where $\zeta\in(0,1)$ is a constant.
Construct the new connected region $\tilde{\bf U}={\bf U}\cup
C_{(\bar{\bf x}_0,x_n)}$ and define the function $F_{\tilde{\bf U}}$
similarly as in (\ref{fdef}).
\begin{assumption}\label{unif1}
For any point $(\bar{\bf x}_0,x_n)\in\bar{\bf U}$ with
$x_n<a(\bar{\bf x}_0)$, $F_{\tilde{\bf U}}(\bar{\bf
x}_0,x_n)\geqslant F_{\bf U}(\bar{\bf x}_0,x_n)$.
\end{assumption}

Similarly, if we consider any point $(\bar{\bf x}_0,x_n)\in\bar{\bf
U}$ with $x_n>b(\bar{\bf x}_0)$, we may define the cone
\begin{displaymath}C_{(\bar{\bf x}_0,x_n)}=\left\{(\bar{\bf
x},y):x_n\geqslant y\geqslant b(\bar{\bf
x}),\frac{x_n-y}{\|<\bar{\bf x}-\bar{\bf x}_0,y-x_n>\|_2}\geqslant
\zeta\right\},\end{displaymath} where $\zeta\in(0,1)$ is a constant.
Construct the new connected region $\tilde{\bf U}={\bf U}\cup
C_{(\bar{\bf x}_0,x_n)}$ and define the function $F_{\tilde{\bf U}}$
similarly as in (\ref{fdef}), then we put the following assumption.
\begin{assumption}\label{unif2}
For any point $(\bar{\bf x}_0,x_n)\in\bar{\bf U}$ with
$x_n>b(\bar{\bf x}_0)$, $F_{\tilde{\bf U}}(\bar{\bf
x}_0,x_n)\leqslant F_{\bf U}(\bar{\bf x}_0,x_n)$.
\end{assumption}

\begin{remark}
In the case of one dimensional diffusion, Assumptions \ref{assumpHf}
and \ref{assumAB} are sufficient to imply that $E$ - the
continuation region - is an interval, but in the multidimensional
case, more conditions are needed to guarantee a regular property.
\end{remark}

\begin{proposition}
Under Assumptions \ref{unif1},\ref{unif2}, the continuation region
$E$ is simply connected. Furthermore, the lower and upper boundaries
of $E$ (in $x_n$) are uniformly Lipschitz functions of $\bar{\bf
x}\in{\mathbb R}^{n-1}$.
\end{proposition}
\begin{proof} We have shown that $E$ is connected. For any point $(\bar{\bf
x},x_n)\in\bar{E}$ with $x_n<a(\bar{\bf x})$, we consider the new
continuation region $E\cup C_{(\bar{\bf x},x_n)}$, and have
$F_{E\cup C_{(\bar{\bf x},x_n)}}(\bar{\bf x},x_n)\geqslant
V(\bar{\bf x},x_n)$ by Assumption \ref{unif1}. Since in the region
below the curve $a$, only player $P_2$ wants to stop the game who
wants to maximize his payoff, the new continuation region $E\cup
C_{(\bar{\bf x},x_n)}$ is certainly better than $E$, if not
identical, and this is a contradiction since $E$ is assumed to be
optimal. Thus $E\cup C_{(\bar{\bf x},x_n)}=E$. Since this holds for
any point $(\bar{\bf x},x_n)\in\bar{E}$ with $x_n<a(\bar{\bf x})$,
we know that the region in $E$ below the curve $a$ is simply
connected, and the lower boundary of $E$ is uniformly Lipshcitz
continuous. The second half is proved in a similar manner.
\end{proof}

Now we can claim that $\tilde{E}_1=E_1,\tilde{E}_2=E_2$. In what
follows we shall still use $\tilde{a}$ as the upper boundary (in
$x_n$) of the set $E_1$ (also the lower boundary of $E$), and
$\tilde{b}$ as the lower boundary (in $x_n$) of the set $E_2$ (also
the upper boundary of $E$), and they are uniformly Lipschitz
continuous and bounded curves. We notice that the curve $a$ in
Assumption \ref{assumpHf} is not necessarily identical to the curve
$\tilde{a}$, and the curve $b$ is not necessarily identical to the
curve $\tilde{b}$ either. Figure \ref{region} illustrates the
continuation region $E$ in a two dimensional case.

\begin{figure}[h!]
\centerline{\includegraphics[height=2.5in,
width=2.5in]{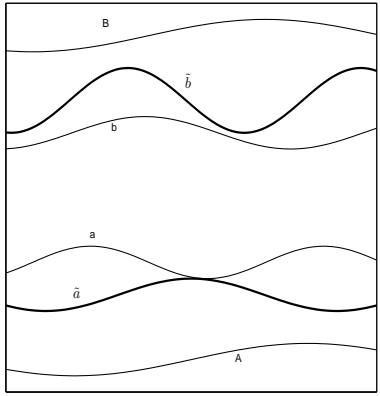}}\caption{The continuation region
between two curves $\tilde{a}(t),\tilde{b}(t)$} \label{region}
\end{figure}

\subsection{Regularities of the Value Function}

\

\begin{proposition}\label{prop2} Assuming Assumptions
\ref{assumpHf}, \ref{assumAB},  \ref{unif1} and \ref{unif2}. If the
curves $\tilde{a}=\tilde{a}(\bar{\bf x})$,
$\tilde{b}=\tilde{b}(\bar{\bf x})$ are smooth, then
 $V$ is
 smooth on ${\mathbb{R}}^{n-1}\times(\tilde{a},\tilde{b})$, and
\begin{eqnarray}
\alpha V({\bf x})-\mathcal{L}V({\bf x})&=H({\bf x}),\quad {\bf
x}\in{\mathbb{R}}^{n-1}\times(\tilde{a},\tilde{b}),\label{inAB}\\
\alpha V({\bf x})-\mathcal{L}V({\bf x})&>H({\bf x}),\quad {\bf
x}\in{\mathbb{R}}^{n-1}\times(-\infty,\tilde{a}), \\
\alpha V({\bf x})-\mathcal{L}V({\bf x})&<H({\bf x}),\quad {\bf
x}\in{\mathbb{R}}^{n-1}\times(\tilde{b},\infty).
\end{eqnarray}
\end{proposition}
\begin{remark}
It is still not clear under what conditions the free boundary curves
 are smooth. But it is reasonable to believe that any condition on
this should involve jointly the underlying process and all the
reward functions of this problem.
\end{remark}
\begin{proof}
We notice that $V$ is $H$-$\alpha$ harmonic on
${\mathbb{R}}^{n-1}\times(\tilde{a},\tilde{b})$, which by
\cite{Fuku11} implies the validation of the following equation:
\begin{equation}\label{VIeq}
\mathcal{E}_\alpha(V,u)=(H,u),\quad \forall u\in
C_0^{1,\cdots,1}({\mathbb{R}}^{n-1}\times(\tilde{a},\tilde{b})).
\end{equation} The continuity of $H$ implies that $V$ is smooth on the
same region, and an integration by parts yields (\ref{inAB}) on this
region. The rest of the proof follows from the fact that $V({\bf
x})=-f_1({\bf x}),\forall {\bf
x}\in{\mathbb{R}}^{n-1}\times(-\infty,\tilde{a})$ and $V({\bf
x})=f_2({\bf x}),\forall {\bf x}\in(\tilde{b},\infty)$.
\end{proof}

In what follows we will characterize the regularity of $V$ on the
boundary curves.

Let $\partial_{\bf u} V$ denote the one-sided directional derivative
along a unit vector ${\bf u}\in{\mathbb{R}}^n$ defined in the
following manner
\begin{displaymath}
\partial_{\bf u} V({\bf x})=\lim_{\lambda\to 0+}\frac{V({\bf x}+\lambda{\bf u})-V({\bf
x})}{\lambda}.
\end{displaymath}
By Proposition \ref{prop2}, it can be seen that $\partial_{\bf u}
V({\bf x})$ is well defined at any point ${\bf x}$ in any direction
${\bf u}$, and $\partial_{\bf u} V({\bf x})$ is continuous. If ${\bf
x}$ is not on the curves $\tilde{a}(\cdot)$ or $\tilde{b}(\cdot)$,
then obviously $\partial_{\bf u} V({\bf x})=\partial_{-\bf u} V({\bf
x})$. The following proposition will characterize the property of
$\partial_{\bf u} V({\bf x})$ when ${\bf x}$ is on
$\tilde{a}(\cdot)$ or $\tilde{b}(\cdot)$.

\begin{proposition} Assuming Assumptions \ref{assumpHf},
\ref{assumAB}, \ref{unif1} and \ref{unif2}. If the curves
$\tilde{a}(\bar{\bf x})$, $\tilde{b}(\bar{\bf x})$ are smooth, then
\begin{displaymath}
\partial_{\bf u} V(\bar{\bf x},\tilde{a}(\bar{\bf
x}))=-\partial_{\bf u} f_{1}(\bar{\bf x},\tilde{a}(\bar{\bf x})), \
\partial_{\bf u} V(\bar{\bf x},\tilde{b}(\bar{\bf x}))=\partial_{\bf u} f_{2}(\bar{\bf x},\tilde{b}(\bar{\bf x})),\ \forall
\bar{\bf x}\in{\mathbb{R}}^{n-1},\ \forall {\bf u}\in{\mathbb{R}}^n.
\end{displaymath}
\end{proposition}

\begin{proof}
\begin{itemize}
\item[Step 1:]
Pick any point $(\bar{\bf x}_0,\tilde{a}(\bar{\bf x}_0))$ on the
curve $\tilde{a}(\cdot)$ and any unit vector ${\bf u}\in{\mathbb
R}^n$ (similar result can be derived for the curve
$\tilde{b}(\cdot)$), and construct a ball $B_\delta(\bar{\bf
x}_0,\tilde{a}(\bar{\bf x}_0))$ centered at this point with radius
$\delta$, and such that for any point $(\bar{\bf x},x_n)\in
B_\delta(\bar{\bf x}_0,\tilde{a}(\bar{\bf x}_0))$, we have
$A({\bar{\bf x}})<x_n<b(\bar{\bf x})$. This can be easily done under
Assumptions \ref{assumpHf} and \ref{assumAB}. Also by the smoothness
and Lipschitz property of the curve $\tilde{a}(\cdot)$, we can
choose $\delta$ small so that all the points $(\bar{\bf
x}_0,\tilde{a}(\bar{\bf x}_0))+\lambda {\bf u}$
($0<\lambda\leqslant\delta$) are either on the curve
$\tilde{a}(\cdot)$, below the curve $\tilde{a}(\cdot)$ or above the
curve $\tilde{a}(\cdot)$. In the first two cases, since we have
showed that $V=-f_1$, the result automatically holds. Now we assume
$(\bar{\bf x}_0,\tilde{a}(\bar{\bf x}_0))+\lambda {\bf u}$
($0<\lambda\leqslant\delta$) are above the curve $\tilde{a}(\cdot)$,
then it has been shown that $V((\bar{\bf x}_0,\tilde{a}(\bar{\bf
x}_0))+\lambda {\bf u})>-f_1((\bar{\bf x}_0,\tilde{a}(\bar{\bf
x}_0))+\lambda {\bf u})$, hence $\partial_{\bf u}V(\bar{\bf
x}_0,\tilde{a}(\bar{\bf x}_0))\geqslant -\partial_{\bf
u}f_1(\bar{\bf x}_0,\tilde{a}(\bar{\bf x}_0))$. Assume the equality
does not hold, then there is $\epsilon>0$ such that $\partial_{\bf
u}V(\bar{\bf x}_0,\tilde{a}(\bar{\bf x}_0))= -\partial_{\bf
u}f_1(\bar{\bf x}_0,\tilde{a}(\bar{\bf x}_0))+\epsilon$. By the
continuity of $\partial_{\bf u}V$ and the Lipschitz property of
$\tilde{a}(\cdot)$, we can construct a cone $\mathcal{C}_{\bf u}$
containing ${\bf u}$ such that $\forall {\bf v}\in \mathcal{C}_{\bf
u}$, the points $(\bar{\bf x}_0,\tilde{a}(\bar{\bf x}_0))+\lambda
{\bf v}$ ($0<\lambda\leqslant\delta$) are above the curve
$\tilde{a}(\cdot)$ and $\partial_{\bf v}V(\bar{\bf
x}_0,\tilde{a}(\bar{\bf x}_0))>-\partial_{\bf v}f_1(\bar{\bf
x}_0,\tilde{a}(\bar{\bf x}_0))+\epsilon/2$.

\item[Step 2:] Define $\tau_{\tilde{a},\delta}$ the first exit time from $B_\delta(\bar{\bf
x}_0,\tilde{a}(\bar{\bf x}_0))$. Firstly, by the fact that
$\bf{X}_t$ is uniformly elliptic, we have
\begin{displaymath}E_{(\bar{\bf x}_0,\tilde{a}(\bar{\bf
x}_0))}(\tau_{\tilde{a},\delta})=o(\delta),
\end{displaymath}where $o(\delta)$ is a small quantity satisfying $\lim_{\delta\to 0}{o(\delta)}/{\delta}=0$, (see, e.g., \cite{Fuku02}, \cite{Karatzas91}). Further, we
notice that
\begin{displaymath}
1-E_{(\bar{\bf x}_0,\tilde{a}(\bar{\bf
x}_0))}\left(e^{-\alpha\tau_{\tilde{a},\delta}}\right)=\alpha
E_{(\bar{\bf x}_0,\tilde{a}(\bar{\bf
x}_0))}\left(\int_0^{\tau_{\tilde{a},\delta}}e^{-\alpha
t}dt\right)\leqslant\alpha E_{(\bar{\bf x}_0,\tilde{a}(\bar{\bf
x}_0))}(\tau_{\tilde{a},\delta})=o(\delta).
\end{displaymath}

\item[Step 3:] Recall that player $P_2$ would stop the
game at the point $(\bar{\bf x}_0,\tilde{a}(\bar{\bf x}_0))$
immediately, and the payoff is $-f_1(\bar{\bf
x}_0,\tilde{a}(\bar{\bf x}_0))$. Consider the region
$B_\delta(\bar{\bf x}_0,\tilde{a}(\bar{\bf x}_0))$, it is shown that
play $P_1$ would not stop the game in this region. Since
$\hat{\tau}=\tau_{\tilde{a},\delta}+\hat{\tau} \circ
\theta_{\tau_{\tilde{a},\delta}}$ where $\theta_t$ is the shift
operator, we get
\begin{displaymath}
\hat{\tau}\wedge\sigma=\tau_{\tilde{a},\delta}+(\hat{\tau}\wedge\hat{\sigma})\circ\theta_{\tau_{\tilde{a},\delta}},
\end{displaymath} where
$\sigma=\tau_{\tilde{a},\delta}+\hat{\sigma}\circ\theta_{\tau_{\tilde{a},\delta}}$.
Therefore by (\ref{saddle_Cont}),
\begin{eqnarray}\label{ineqVx0}
&&-f_1(\bar{\bf x}_0,\tilde{a}(\bar{\bf x}_0))=V(\bar{\bf x}_0,\tilde{a}(\bar{\bf x}_0))\geqslant J_{(\bar{\bf x}_0,\tilde{a}(\bar{\bf x}_0))}(\hat{\tau},\sigma)\nonumber\\
&&\ \ \ \ \ \ \ =E_{(\bar{\bf x}_0,\tilde{a}(\bar{\bf
x}_0))}\left(\int_0^{\tau_{\tilde{a},\delta}}e^{-\alpha t}H({\bf
X}_t)dt\right) +E_{(\bar{\bf x}_0,\tilde{a}(\bar{\bf
x}_0))}\left(e^{-\alpha\tau_{\tilde{a},\delta}}V({\bf
X}_{\tau_{\tilde{a},\delta}})\right).
\end{eqnarray}
By Assumption \ref{assumpHf}, it is obvious that $H$ is bounded on
$B_\delta(\bar{\bf x}_0,\tilde{a}(\bar{\bf x}_0))$, i.e., there
exists $U>0$ such that $|H({\bf x})|\leqslant U,\ \forall {\bf x}\in
B_\delta(\bar{\bf x}_0,\tilde{a}(\bar{\bf x}_0))$. Then we have
\begin{eqnarray}
&&0\leqslant \left|E_{(\bar{\bf x}_0,\tilde{a}(\bar{\bf
x}_0))}\left(\int_0^{\tau_{\tilde{a},\delta}}e^{-\alpha t}H({\bf
X}_t)dt\right)\right|\leqslant UE_{(\bar{\bf x}_0,\tilde{a}(\bar{\bf
x}_0))}\left(\int_0^{\tau_{\tilde{a},\delta}}e^{-\alpha
t}dt\right)\nonumber\\
&&= \frac{U}{\alpha}\left(1-E_{(\bar{\bf x}_0,\tilde{a}(\bar{\bf
x}_0))}\left(e^{-\alpha\tau_{\tilde{a},\delta}}\right)\right)
=o(\delta).\nonumber
\end{eqnarray}

The other part
\begin{eqnarray}
&&E_{(\bar{\bf x}_0,\tilde{a}(\bar{\bf
x}_0))}\left(e^{-\alpha\tau_{\tilde{a},\delta}}V({\bf
X}_{\tau_{{a},\delta}})\right)=E_{(\bar{\bf x}_0,\tilde{a}(\bar{\bf
x}_0))}\left((1-\alpha\tau_{\tilde{a},\delta}+o(\alpha\tau_{\tilde{a},\delta}))V({\bf
X}_{\tau_{\tilde{a},\delta}})\right)\nonumber\\
&&=E_{(\bar{\bf x}_0,\tilde{a}(\bar{\bf x}_0))}\left(V({\bf
X}_{\tau_{\tilde{a},\delta}})\right)+E_{(\bar{\bf
x}_0,\tilde{a}(\bar{\bf
x}_0))}\left((-\alpha\tau_{\tilde{a},\delta}+o(\alpha\tau_{\tilde{a},\delta}))V({\bf
X}_{\tau_{\tilde{a},\delta}})\right).\nonumber
\end{eqnarray} Since $V$ is also bounded and $E_{(\bar{\bf x}_0,\tilde{a}(\bar{\bf
x}_0))}(\tau_{\tilde{a},\delta})=o(\delta)$, we get
\begin{displaymath}
E_{(\bar{\bf x}_0,\tilde{a}(\bar{\bf
x}_0))}\left(e^{-\alpha\tau_{\tilde{a},\delta}}V({\bf
X}_{\tau_{\tilde{a},\delta}})\right)=E_{(\bar{\bf
x}_0,\tilde{a}(\bar{\bf x}_0))}\left(V({\bf
X}_{\tau_{\tilde{a},\delta}})\right)+o(\delta).
\end{displaymath}

\item[Step 4:] By Dynkin's formula,
\begin{displaymath}
E_{(\bar{\bf x}_0,\tilde{a}(\bar{\bf x}_0))}\left(-f_1({\bf
X}_{\tau_{\tilde{a},\delta}})\right)=-f_1(\bar{\bf
x}_0,\tilde{a}(\bar{\bf x}_0))+E_{(\bar{\bf x}_0,\tilde{a}(\bar{\bf
x}_0))}\left(\int_0^{\tau_{\tilde{a},\delta}}\mathcal{L}(-f_1)({\bf
X}_t)dt\right).
\end{displaymath} Since the function $-f_1$ is bounded on $B_\delta(\bar{\bf x}_0,\tilde{a}(\bar{\bf
x}_0))$, and $E_{(\bar{\bf x}_0,\tilde{a}(\bar{\bf
x}_0))}(\tau_{\tilde{a},\delta})=o(\delta)$, we get
\begin{equation}\label{f_1approx}
E_{(\bar{\bf x}_0,\tilde{a}(\bar{\bf x}_0))}\left(-f_1({\bf
X}_{\tau_{\tilde{a},\delta}})\right)=-f_1(\bar{\bf
x}_0,\tilde{a}(\bar{\bf x}_0))+o(\delta).
\end{equation}
On the other hand, for each point ${\bf y}\in \partial
B_\delta(\bar{\bf x}_0,\tilde{a}(\bar{\bf x}_0))$, i.e., the
boundary of $B_\delta(\bar{\bf x}_0,\tilde{a}(\bar{\bf x}_0))$, and
by the smoothness of $-f_1$,
\begin{displaymath}
-f_1({\bf y})=-f_1(\bar{\bf x}_0,\tilde{a}(\bar{\bf
x}_0))+\partial_{\bf v}(-f_1)(\bar{\bf x}_0,\tilde{a}(\bar{\bf
x}_0))\delta+o(\delta),
\end{displaymath} so
\begin{eqnarray}\label{f_1approx1}
&&E_{(\bar{\bf x}_0,\tilde{a}(\bar{\bf x}_0))}\left(-f_1({\bf
X}_{\tau_{\tilde{a},\delta}})\right)\nonumber\\
&&=-f_1(\bar{\bf x}_0,\tilde{a}(\bar{\bf x}_0))+\int\partial_{\bf
v}(-f_1)(\bar{\bf x}_0,\tilde{a}(\bar{\bf x}_0))\delta p_{(\bar{\bf
x}_0,\tilde{a}(\bar{\bf x}_0))}(d{\bf y})+o(\delta),
\end{eqnarray} where $p_{(\bar{\bf
x}_0,\tilde{a}(\bar{\bf x}_0))}(d{\bf y})$ is the probability
density function of the exit distribution. By comparing
(\ref{f_1approx}) with (\ref{f_1approx1}), we get
\begin{displaymath}
\int\partial_{\bf v}(-f_1)(\bar{\bf x}_0,\tilde{a}(\bar{\bf
x}_0))\delta p_{(\bar{\bf x}_0,\tilde{a}(\bar{\bf x}_0))}(d{\bf
y})=o(\delta).
\end{displaymath}

\item[Step 5:] On the boundary of $B_\delta(\bar{\bf x}_0,\tilde{a}(\bar{\bf x}_0))$,
$V({\bf y})\geqslant -f_1({\bf y})$. Since $V$ is smooth a.e.,
\begin{displaymath}
V({\bf y})=V(\bar{\bf x}_0,\tilde{a}(\bar{\bf x}_0))+\partial_{\bf
v}V(\bar{\bf x}_0,\tilde{a}(\bar{\bf x}_0))\delta+o(\delta),
\end{displaymath} and
\begin{eqnarray}
&&E_{(\bar{\bf x}_0,\tilde{a}(\bar{\bf x}_0))}\left(V({\bf
X}_{\tau_{\tilde{a},\delta}})\right)\\
&&=V(\bar{\bf x}_0,\tilde{a}(\bar{\bf x}_0))+\int \partial_{\bf
v}V(\bar{\bf x}_0,\tilde{a}(\bar{\bf x}_0))\delta p_{(\bar{\bf
x}_0,\tilde{a}(\bar{\bf x}_0))}(d{\bf y})+o(\delta)\nonumber\\
&&=-f_1(\bar{\bf x}_0,\tilde{a}(\bar{\bf x}_0))+\int \partial_{\bf
v}V(\bar{\bf x}_0,\tilde{a}(\bar{\bf x}_0))\delta p_{(\bar{\bf
x}_0,\tilde{a}(\bar{\bf x}_0))}(d{\bf y})+o(\delta)\nonumber\\
&&\geqslant -f_1(\bar{\bf x}_0,\tilde{a}(\bar{\bf x}_0))+\int
\partial_{\bf v}(-f_1)(\bar{\bf x}_0,\tilde{a}(\bar{\bf x}_0))\delta
p_{(\bar{\bf x}_0,\tilde{a}(\bar{\bf x}_0))}(d{\bf
y})\nonumber\\
&& + \int_{\mathcal{C}_{\bf u}\cap\partial B_\delta(\bar{\bf
x}_0,\tilde{a}(\bar{\bf x}_0))} \frac{\epsilon}{2}\delta
p_{(\bar{\bf x}_0,\tilde{a}(\bar{\bf x}_0))}(d{\bf
y})-|o(\delta)|,\nonumber
\end{eqnarray} in view of the conclusion of Step 1, and $\mathcal{C}_{\bf u}\cap\partial B_\delta(\bar{\bf
x}_0,\tilde{a}(\bar{\bf x}_0))$ is the part of $\partial
B_\delta(\bar{\bf x}_0,\tilde{a}(\bar{\bf x}_0))$ in the cone
$\mathcal{C}_{\bf u}$. Substitute in (\ref{ineqVx0}) the results
from Steps 3 and 4, we get
\begin{eqnarray}
&&-f_1(\bar{\bf x}_0,\tilde{a}(\bar{\bf x}_0))=V(\bar{\bf
x}_0,\tilde{a}(\bar{\bf x}_0))\nonumber\\
&& \geqslant -f_1(\bar{\bf x}_0,\tilde{a}(\bar{\bf
x}_0))+\int_{\mathcal{C}_{\bf u}\cap\partial B_\delta(\bar{\bf
x}_0,\tilde{a}(\bar{\bf x}_0))} \frac{\epsilon}{2}\delta
p_{(\bar{\bf x}_0,\tilde{a}(\bar{\bf x}_0))}(d{\bf
y})-|o(\delta)|.\nonumber
\end{eqnarray}
Since ${\bf X}_t$ is uniformly elliptic, $\int_{\mathcal{C}_{\bf
u}\cap\partial B_\delta(\bar{\bf x}_0,\tilde{a}(\bar{\bf x}_0))}
 p_{(\bar{\bf x}_0,\tilde{a}(\bar{\bf
x}_0))}(d{\bf y})\to \eta$ as $\delta\to 0$, where $0<\eta<1$ is a
constant. By choosing $\delta$ small, we get
\begin{eqnarray}
-f_1(\bar{\bf x}_0,\tilde{a}(\bar{\bf x}_0))=V(\bar{\bf
x}_0,\tilde{a}(\bar{\bf x}_0))\geqslant -f_1(\bar{\bf
x}_0,\tilde{a}(\bar{\bf
x}_0))+\frac{\epsilon\delta\eta}{4}-|o(\delta)|.
\end{eqnarray} If $\delta$ is sufficiently small, the part
$\frac{\epsilon\delta\eta}{4}-|o(\delta)|>0$, and we get a
contradiction.
\end{itemize}
Now the proof is complete.
\end{proof}

As a summary we have the following theorem:
\begin{theorem}\label{Vcond} Assuming Assumptions \ref{assumpHf},
\ref{assumAB}, \ref{unif1} and \ref{unif2}. If the curves
$\tilde{a}(\bar{\bf x})$, $\tilde{b}(\bar{\bf x})$ are smooth, then
\begin{eqnarray*}
-f_1({\bf x})<&&V({\bf x})<f_2({\bf x}),\quad \forall{\bf x}\in
{\mathbb{R}}^{n-1}\times(\tilde{a},\tilde{b}),\\
V({\bf x})&&=-f_1({\bf x}), \quad \forall{\bf x}\in {\mathbb{R}}^{n-1}\times(-\infty,\tilde{a}],\\
V({\bf x})&&=f_2({\bf x}), \quad \forall{\bf x}\in
{\mathbb{R}}^{n-1}\times[\tilde{b},\infty),
\end{eqnarray*}
\begin{displaymath}
\partial_{\bf u} V(\bar{\bf x},\tilde{a}(\bar{\bf x}))=-\partial_{\bf u}
f_{1}(\bar{\bf x},\tilde{a}(\bar{\bf x}),\ \partial_{\bf u}
V(\bar{\bf x},\tilde{b}(\bar{\bf x}))=\partial_{\bf u}
f_{2}(\bar{\bf x},\tilde{b}(\bar{\bf x})),\ \forall \bar{\bf
x}\in{\mathbb{R}}^{n-1},
\end{displaymath} where ${\bf u}$ is any directional vector.

Furthermore $V$ is $C^{1,\cdots,1,1}$ on ${\mathbb{R}}^{n}$,
$C^{2,\cdots,2}$ on
${\mathbb{R}}^{n-1}\times(\tilde{a},\tilde{b})\cup{\mathbb{R}}^{n-1}\times(-\infty,\tilde{a})\cup{\mathbb{R}}^{n-1}\times(\tilde{b},\infty)$
and
\begin{eqnarray*}
\alpha V({\bf x})-\mathcal{L}V({\bf x})&=H({\bf x}),\quad
 &\forall{\bf x}\in {\mathbb{R}}^{n-1}\times(\tilde{a},\tilde{b}), \\
\alpha V({\bf x})-\mathcal{L}V({\bf x})&>H({\bf x}),\quad
 &\forall{\bf x}\in {\mathbb{R}}^{n-1}\times(-\infty,\tilde{a}), \\
\alpha V({\bf x})-\mathcal{L}V({\bf x})&<H({\bf x}),\quad
&\forall{\bf x}\in {\mathbb{R}}^{n-1}\times(\tilde{b},\infty),
\end{eqnarray*} where $\mathcal{L}$ is given in (\ref{infgen}).
\end{theorem}

\section{The Multi-Dimensional Stochastic Singular Control
Problem}\label{mdssc}
Define $h({\bf x}),W({\bf x}),\ {\bf x}\in{\mathbb{R}}^n$, as
follows:
\begin{eqnarray}
h(\bar{\bf x},y)&=&\int_{\tilde{a}(\bar{\bf x})}^yH(\bar{\bf x},u)du+C(\bar{\bf x}),\label{hcost}\\
W(\bar{\bf x},y)&=&\int_{\tilde{a}(\bar{\bf x})}^yV(\bar{\bf
x},u)du,\quad \bar{\bf x}\in{\mathbb{R}}^{n-1},\
y\in{\mathbb{R}},\label{wval}
\end{eqnarray}
where $C(\bar{\bf x})$ is a function of $\bar{\bf x}$ such that
\begin{displaymath}
\lim_{y\to \tilde{a}(\bar{\bf x})+}\alpha W(\bar{\bf
x},y)-\mathcal{L}W(\bar{\bf x},y)-h(\bar{\bf x},y)=0,
\end{displaymath}
 then
$h(\bar{\bf x},y)$ and $W(\bar{\bf x},y)$ satisfy the following:

\begin{theorem}\label{Wcond} Assuming Assumptions \ref{assumpHf},
\ref{assumAB}, \ref{unif1} and \ref{unif2}. If the curves
$\tilde{a}(\bar{\bf x})$, $\tilde{b}(\bar{\bf x})$ are smooth, then
$W$ is $C^{2,\cdots,2}$ on ${\mathbb{R}}^n$ and
\begin{eqnarray*}
&&\alpha W({\bf x})-\mathcal{L}W({\bf x})=h({\bf x}),\quad
 \forall{\bf x}\in {\mathbb{R}}^{n-1}\times(\tilde{a},\tilde{b}), \\
&&\alpha W({\bf x})-\mathcal{L}W({\bf x})<h({\bf x}),\quad
\forall{\bf x}\in {\mathbb{R}}^{n-1}\times(-\infty,\tilde{a}) \cup
{\mathbb{R}}^{n-1}\times(\tilde{b},\infty),\\
&&-f_1({\bf x})<\frac{\partial }{\partial x_n}W({\bf x})<f_2({\bf x}),\quad \forall{\bf x}\in {\mathbb{R}}^{n-1}\times(\tilde{a},\tilde{b}),\\
&&\frac{\partial}{\partial x_n}W({\bf x})=-f_1({\bf x}), \quad \forall{\bf x}\in {\mathbb{R}}^{n-1}\times(-\infty,\tilde{a}],\\
&&\frac{\partial}{\partial x_n}W({\bf x})=f_2({\bf x}), \quad
\forall{\bf x}\in {\mathbb{R}}^{n-1}\times[\tilde{b},\infty),
\end{eqnarray*} and
\begin{eqnarray*}
&&\frac{\partial^2}{\partial x_n\partial x_k}W(\bar{\bf
x},\tilde{a}(\bar{\bf x}))=-\frac{\partial f_1}{\partial
x_k}(\bar{\bf
x},\tilde{a}(\bar{\bf x})),\nonumber\\
&&\frac{\partial^2}{\partial x_n\partial x_k}W(\bar{\bf
x},\tilde{b}(\bar{\bf x}))=\frac{\partial f_2}{\partial
x_k}(\bar{\bf x},\tilde{b}(\bar{\bf x})),\ \forall\bar{\bf
x}\in{\mathbb{R}}^{n-1},\ 1\leqslant k\leqslant n.\nonumber
\end{eqnarray*}
\end{theorem}
We first need a lemma.

\begin{lemma}\label{cont_alpha_L} Assuming Assumptions
\ref{assumpHf}, \ref{assumAB}, \ref{unif1} and \ref{unif2} and  that
the curves $\tilde{a}(\bar{\bf x})$, $\tilde{b}(\bar{\bf x})$ being
smooth. The function $\alpha W({\bf x})-\mathcal{L}W({\bf x})$ is
continuous.
\end{lemma}
\begin{proof}
This result obviously holds for ${\bf
x}\in{\mathbb{R}}^{n-1}\times(-\infty,\tilde{a})\cup{\mathbb{R}}^{n-1}\times(\tilde{a},\tilde{b})\cup{\mathbb{R}}^{n-1}\times(\tilde{b},\infty)$.
On the curves $\tilde{a}(\bar{\bf x})$ and $\tilde{b}(\bar{\bf x})$,
$W$ is twice continuously differentiable along the $x_n$ direction
by (\ref{wval}). The only term that seems not to be continuous in
this function is $\mathcal{L}W({\bf x})$, which involves the first
and second derivative with respect to each variable.  Denote
$\partial_{x_k} W$ the directional derivative along $x_k, 1\leqslant
k\leqslant n-1$, then by (\ref{wval}) we have the following:
\begin{displaymath}\partial_{x_k}W(\bar{\bf x},y)=\int_{\tilde{a}(\bar{\bf
x})}^y \partial_{x_k}V(\bar{\bf x},u)du-V(\bar{\bf
x},\tilde{a}(\bar{\bf x}))\partial_{x_k}\tilde{a}(\bar{\bf x}).
\end{displaymath} Notice that  $V$ is
$C^{1,1,...,1}$ on ${\mathbb{R}}^n$, so $\partial_{x_k}W(\bar{\bf
x},y)=\partial_{x_k}W({\bf x})$ is continuous in ${\bf x}$. Now
consider
\begin{eqnarray*}
\frac{\partial^2 W}{\partial x_k^2}(\bar{\bf
x},y)&=&\int_{\tilde{a}(\bar{\bf x})}^y\frac{\partial^2 V}{\partial
x_k^2}(\bar{\bf x},u)du-2\partial_{x_k}V(\bar{\bf
x},\tilde{a}(\bar{\bf x}))\cdot\partial_{x_k}\tilde{a}(\bar{\bf
x})\\
\quad && -\partial_{x_n}V(\bar{\bf x},\tilde{a}(\bar{\bf
x}))\cdot(\partial_{x_k}\tilde{a}(\bar{\bf x}))^2-V(\bar{\bf
x},\tilde{a}(\bar{\bf x}))\frac{\partial^2 \tilde{a}}{\partial
x_k^2}(\bar{\bf x}).
\end{eqnarray*} Because the two curves $\tilde{a}(\bar{\bf x})$ and $\tilde{b}(\bar{\bf
x})$ have zero Lebesgue measure, and the functions
$\partial_{x_k}V$, $\partial_{x_k}\tilde{a}$ and $\frac{\partial^2
\tilde{a}}{\partial x_k^2}$ are all continuous, we conclude that
$\frac{\partial^2 W}{\partial x_k^2}$ is continuous, $1\leqslant
k\leqslant n-1$. The continuity of $\frac{\partial^2 W}{\partial
x_i\partial x_j},i\neq j,$ can be proved in a similar manner.
Combined with  previous argument that $\frac{\partial^2 W}{\partial
x_n^2}$ is continuous, this lemma is proved.
\end{proof}
\begin{remark}
Since $(\alpha-\mathcal{L})W$ is continuous, and the functions $V$
and $H$ are continuous too, we know that the function $h({\bf x})$
in (\ref{hcost}) is continuous, hence the continuity of $C(\bar{\bf
x})$ in (\ref{hcost}).
\end{remark}

\begin{proof}[Proof of Theorem \ref{Wcond}]
For fixed $\bar{\bf x}$, consider the function
\begin{displaymath}
U(y)=\alpha W(\bar{\bf x},y)-\mathcal{L}W(\bar{\bf x},y)-h(\bar{\bf
x},y)
\end{displaymath}
with
\begin{displaymath}
U'(y)=\alpha V(\bar{\bf x},y)-\mathcal{L}V(\bar{\bf x},y)-H(\bar{\bf
x},y),
\end{displaymath} and we know $U(\tilde{a}(\bar{\bf x}))=0$. Notice that $U'(y)=0$ for $\tilde{a}(\bar{\bf x})< y< \tilde{b}(\bar{\bf
x})$; $U'(y)>0$ for $y<\tilde{a}(\bar{\bf x})$; $U'(y)<0$ for
$y>\tilde{b}(\bar{\bf x})$, and by Lemma \ref{cont_alpha_L} the
function $U(y)$ is continuous, it can be seen that
\begin{displaymath}
\alpha W(\bar{\bf x},y)-\mathcal{L}W(\bar{\bf x},y)<h(\bar{\bf
x},y), {\rm\ for\ } y<\tilde{a}(\bar{\bf x}) {\rm\ or\ }
y>\tilde{b}(\bar{\bf x}).
\end{displaymath} The rest of the proof is obvious.
\end{proof}

The result of Theorem \ref{Wcond} gives conditions to the solution
of the stochastic singular control problem (\ref{scost}) and
(\ref{mincostw}) (see, e.g., \cite{Taksar85}), where the holding
cost $h(\cdot)$ is given in (\ref{hcost}) and the boundary penalty
costs $f_1(\cdot),f_2(\cdot)$ are given in Assumption
\ref{assumpHf}.

We call a quadruplet $\mathcal{S}=(S,{\bf X}_t,A_t^{(1)},A_t^{(2)})$
($\mathcal{S}=(A_t^{(1)},A_t^{(2)})$ for simplicity) admissible
policy if the following conditions are satisfied:

\begin{assumption}\label{assumpgen}
\begin{enumerate}

\item[{[}1{]}] $S$ is a compact region given in the form
${\mathbb{R}}^{n-1}\times[\beta,\gamma]$ where $\beta(\bar{\bf x}),
\gamma(\bar{\bf x})$ are continuous functions of $\bar{\bf
x}\in{\mathbb{R}}^{n-1}$ with  $\beta(\bar{\bf x})< \gamma(\bar{\bf
x})$.

\item[{[}2{]}] There is a filtered measurable space $(\Omega,\{\mathcal{F}_t\}_{t\geqslant
0})$ subject to usual conditions and a probability measure $\{P_{\bf
x}\}_{{\bf x}\in S}$ on it such that

\begin{itemize}

\item[] $\{{\bf X}_t\}_{t\geqslant 0}$ is an
$\{\mathcal{F}_t\}$-adapted process, and

\item[] $\{A_t^{(1)},A_t^{(2)}\}_{t\geqslant 0}$ are $\{\mathcal{F}_t\}$-adapted right continuous
processes with bounded variation such that
\begin{equation}\label{Acond}
E_{\bf x}\left(\int_{0^-}^\infty e^{-\alpha
t}dA_t^{(1)}\right)<\infty, E_{\bf x}\left(\int_{0^-}^\infty
e^{-\alpha t}dA_t^{(2)}\right)<\infty,\forall {\bf x}\in S,
\end{equation} and $A_t^{(1)}-A_t^{(2)}$ is the minimal decomposition
of a bounded variation process into a difference of two increasing
processes.
\end{itemize}

\item[{[}3{]}] There are $\{\mathcal{F}_t\}$-adapted independent Brownian
motions $B_{1t},...,B_{mt}$ ($m\geqslant n$) starting at the origin
under $P_{\bf x}$ for any ${\bf x}\in S$ such that the following
controlled diffusion ${\bf X}_t=(X_{1t},...,X_{nt})$
\begin{eqnarray}\label{cdiff}
dX_{1t}&=&\mu_1dt+\sigma_{11}dB_{1t}+\cdots+\sigma_{1m}dB_{mt},\\
\vdots && \quad\quad \vdots \quad\quad\quad \vdots\nonumber\\
dX_{nt}&=&\mu_ndt+\sigma_{n1}dB_{1t}+\cdots+\sigma_{nm}dB_{mt}+dA_t^{(1)}-dA_t^{(2)},\nonumber\\
{\bf X}_0&=&{\bf x},\nonumber
\end{eqnarray} holds $P_{\bf x}$-a.s.,  $\forall{\bf x}\in
S$. Furthermore we assume
\begin{displaymath}
P_{\bf x}({\bf X}_t\in S,\forall t\geqslant 0)=1,\quad \forall {\bf
x}\in S.
\end{displaymath}
\end{enumerate}
\end{assumption}

\begin{remark} The probability space $\Omega$ with the filtration
$\{\mathcal{F}_t\}$ is not fixed a priori. It is part of an
admissible policy. The filtration $\{\mathcal{F}_t\}$ is assumed to
be right continuous and $\mathcal{F}_0$ is assumed to contain every
$P_{\bf x}$-negligible set for any ${\bf x}\in S$.
\end{remark}

\begin{proposition}
Both $A_t^{(1)}$ and $A_t^{(2)}$ are nontrivial in the sense that
for any $T>0$,
\begin{displaymath}
P_{\bf x}(A_t^{(i)}=A_0^{(i)},\ \ \forall t\in[0,T])=0,\quad \forall
{\bf x}\in S,\ \ i=1,2.
\end{displaymath}
\end{proposition}
\begin{proof}
If both $A_t^{(1)}$ and $A_t^{(2)}$ are trivial, ${\bf X}_t$ will
hit every open region of positive Lebesgue measure in
${\mathbb{R}}^n$ with positive probability, but this is a
contradiction since ${\bf X}_t$ is concentrated on $S$. If either
$A_t^{(1)}$ or $A_t^{(2)}$ is trivial, ${\bf X}_t$ can not be
concentrated on $S$ which again is a contradiction.
\end{proof}

Define the following notations:
\begin{eqnarray*}
\Delta A_t^{(i)}&=&A_t^{(i)}-A_{t^-}^{(i)},\quad t\geqslant 0,
i=1,2,\\
\Delta {\bf X}_t&=&{\bf X}_t-{\bf X}_{t^-},\quad t\geqslant
0,\\
 \Delta W({\bf X}_t)&=&W({\bf X}_t)-W({\bf X}_{t^-}),\quad
t\geqslant 0.
\end{eqnarray*} Then due to the fact that $A_t^{(1)},A_t^{(2)}$
are the minimal decomposition of a bounded variation process into a
difference of two increasing processes, $\Delta A_t^{(1)}\cdot\Delta
A_t^{(2)}=0,\ \forall t\geqslant 0$. By convention we let
\begin{displaymath}
B_{1t}=\cdots=B_{nt}=0,\ A_{t}^{(1)}=A_{t}^{(2)}=0,\quad \forall
t<0,
\end{displaymath}
so that
\begin{displaymath}
\Delta A_0^{(i)}=A_0^{(i)},\quad i=1,2,\quad {\bf X}_0={\bf x}\quad
P_{\bf x} {\rm\ a.s.},\quad {\bf x}\in S.
\end{displaymath}

Notice that the integrals in (\ref{Acond}) involve the possible
jumps at time $0$ so that they are the sum of the integrals over
$(0,\infty)$ as well as $A_0^{(i)},i=1,2$. In particular, the jump
only happens to the $x_n$ coordinate. In what follows, we use
$A_t^{(i),c} (i=1,2)$ to denote the continuous part of the processes
$A_t^{(i)},i=1,2$.

\begin{theorem}\label{Veri}
Assuming Assumptions \ref{assumpHf}, \ref{assumAB}, \ref{unif1},
\ref{unif2}, \ref{assumpgen} and  that the curves
$\tilde{a}(\bar{\bf x})$, $\tilde{b}(\bar{\bf x})$ being smooth. Let
$k_{\mathcal S}({\bf x})$ be given by the following
\begin{eqnarray}\label{kcostjump}
k_{\mathcal{S}}({\bf x})=&&E_{\bf x}\left(\int_0^\infty e^{-\alpha
t}h({\bf X_t})dt\right)\\
&&+E_{\bf x}\left(\int_0^\infty e^{-\alpha t}\left(f_1({\bf
X}_t)dA_t^{(1),c}+f_2({\bf X}_t)dA_t^{(2),c}\right)\right)\nonumber\\
&&+E_{\bf x}\left(\sum_{0\leqslant t<\infty}e^{-\alpha
t}\left(\int_{X_{nt^-}}^{X_{nt^-}+\Delta A_t^{(1)}} f_1({\bf
X}_t)dy\right.\right.\nonumber\\
&&\left.\left.+\int_{X_{nt^-}-\Delta A_t^{(2)}}^{X_{nt^-}} f_2({\bf
X}_t)dy\right)\right),\nonumber
\end{eqnarray}
then
\begin{enumerate}

\item For any admissible policy $\mathcal{S}$, $W({\bf x})\leqslant k_{\mathcal S}({\bf
x}),\ \forall {\bf x}\in{\mathbb{R}}^n$.

\item $W({\bf x})= k_{\mathcal S}({\bf
x}),\ \forall {\bf x}\in{\mathbb{R}}^n$, if and only if
$S={\mathbb{R}}^{n-1}\times[\tilde{a},\tilde{b}]$, where
$\tilde{a}(\bar{\bf x}),\tilde{b}(\bar{\bf x})$ are given in Theorem
\ref{Wcond}, the process ${\bf X}_t$ is the reflecting diffusion on
$S$, and $\mathcal{S}=(A_t^{(1)},A_t^{(2)})$ where $A_t^{(1)}$
increases only when ${\bf X}_t$ is on the boundary $(\bar{\bf
x},\tilde{a}(\bar{\bf x}))$ and $A_t^{(2)}$ increases only when
${\bf X}_t$ is on the boundary $(\bar{\bf x},\tilde{b}(\bar{\bf
x})),\ \forall\bar{\bf x}\in{\mathbb{R}}^{n-1}$.
\end{enumerate}
\end{theorem}

\begin{remark}
The cost function consists of several parts. The first integral in
(\ref{kcostjump}) is the holding cost. The second integral is a
control cost associated with the increment of controls $A_t^{(i)}$
($i=1,2$) in the continuous part. The last integral is a control
cost associated with the jumps in $A_t^{(i)},i=1,2$ (or equivalently
jumps in ${\bf X}_t$). We further extend $k_{\mathcal{S}}({\bf x})$
outside the region ${\mathbb{R}}^{n-1}\times[\beta,\gamma]$ for two
continuous functions $\beta(\bar{\bf x})<\gamma(\bar{\bf x}),\
\forall\bar{\bf x}\in{\mathbb{R}}^{n-1}$ as the following:
\begin{eqnarray}\label{extcost}
k_{\mathcal{S}}({\bf x})&=&k_{\mathcal{S}}(\bar{\bf
x},\beta(\bar{\bf x}))+\int_{x_n}^{\beta(\bar{\bf x})} f_1(\bar{\bf
x},u)du,\quad \forall{\bf
x}=(\bar{\bf x},x_n)\in{\mathbb{R}}^{n-1}\times(-\infty,\beta),\\
k_{\mathcal{S}}({\bf x})&=&k_{\mathcal{S}}(\bar{\bf
x},\gamma(\bar{\bf x}))+\int_{\gamma(\bar{\bf x})}^{x_n}f_2(\bar{\bf
x},u)du,\quad \forall{\bf x}=(\bar{\bf
x},x_n)\in{\mathbb{R}}^{n-1}\times(\gamma,\infty),
\end{eqnarray} and we are looking for an admissible control $\mathcal{S}$ such
that
\begin{equation}\label{winf}
W^*({\bf x})=\inf_{\mathcal{S}\in\mathbb{S}}k_{\mathcal S}({\bf
x}),\quad \forall{\bf x}\in{\mathbb{R}}^n,
\end{equation} where $\mathbb{S}$ is the set of all admissible
control policies.
\end{remark}

\begin{proof}[Proof of Theorem \ref{Veri}]
\begin{enumerate}
\item Consider the diffusion given in (\ref{cdiff}) with ${\bf x}\in S$.
Applying the generalized Ito formula to $e^{-\alpha t}W({\bf X}_t)$
(see \cite{Harrison83}) yields
\begin{eqnarray}\label{gIto}
e^{-\alpha t}W({\bf X_t})&=&W({\bf x})-\alpha\int_0^t e^{-\alpha
s}W({\bf X}_s)ds+\int_0^t e^{-\alpha s}\mathcal{L}W({\bf
X}_s)ds\nonumber\\
&&+\int_0^t e^{-\alpha s}\nabla W({\bf
X}_s) \cdot \sigma({\bf X}_s) d{\bf B}_{s}\\
&&+\int_0^te^{-\alpha s}\frac{\partial }{\partial x_n}W({\bf
X}_s)(dA_{s}^{(1),c}-dA_s^{(2),c})+\sum_{0< s\leqslant t}e^{-\alpha
s}\Delta W({\bf X}_s).\nonumber
\end{eqnarray} Using the following identity
\begin{displaymath}
W({\bf x})+\sum_{0< s\leqslant t}e^{-\alpha s}\Delta W({\bf
X}_s)=W({\bf X}_{0^-})+\sum_{0\leqslant s\leqslant t}e^{-\alpha
s}\Delta W({\bf X}_s),
\end{displaymath}
and taking expectation of both sides of (\ref{gIto}) with respect to
$P_{\bf x}$ and let $t\to\infty$, we get the following:
\begin{eqnarray}\label{widen}
W({\bf x})&=&E_{\bf x}\left(\int_0^\infty e^{-\alpha
t}\left(\alpha-\mathcal{L}\right)W({\bf X}_t)dt\right)\\
&&-E_{\bf x}\left(\int_0^\infty e^{-\alpha
t}\frac{\partial}{\partial x_n}W({\bf
X}_t)(dA_{t}^{(1),c}-dA_t^{(2),c})\right)\nonumber\\
&&-E_{\bf x}\left(\sum_{0\leqslant t <\infty}e^{-\alpha t}\Delta
W({\bf X}_t)\right).\nonumber
\end{eqnarray} Therefore
\begin{eqnarray}\label{kmW}
&&\quad\quad \quad k_{\mathcal S}({\bf x})-W({\bf x})\\
&&=E_{\bf x}\left(\int_0^\infty
 e^{-\alpha t}\left[h({\bf X}_t)-(\alpha-\mathcal{L})W({\bf
X}_t)\right]dt\right)\nonumber\\
&&+E_{\bf x}\left(\int_0^\infty e^{-\alpha t}\left[f_1({\bf
X}_t)+\frac{\partial }{\partial x_n}W({\bf
X}_t)\right]dA_t^{(1),c}\right)\nonumber\\
&&+E_{\bf x}\left(\int_0^\infty e^{-\alpha t}\left[f_2({\bf
X}_t)-\frac{\partial }{\partial x_n}W({\bf X}_t)\right]dA_t^{(2),c}\right)\nonumber\\
&&+E_{\bf x}\left(\sum_{0\leqslant t<\infty}e^{-\alpha t}\Delta
W({\bf X}_t)\right)\nonumber\\
&&+E_{\bf x}\left(\sum_{0\leqslant t<\infty}e^{-\alpha
t}\left(\int_{X_{nt^-}}^{X_{nt^-}+\Delta A_t^{(1)}} f_1({\bf
X}_t)dy+\int_{X_{nt^-}-\Delta A_t^{(2)}}^{X_{nt^-}} f_2({\bf
X}_t)dy\right)\right).\nonumber
\end{eqnarray}
By Theorem \ref{Wcond}, the first three integrands in (\ref{kmW})
are all nonnegative for the process ${\bf X}_t$ staying in the
region $S$.

Define the sets \begin{displaymath}\Gamma_+=\{t\geqslant 0: \Delta
A_t^{(1)}>0\},\quad \Gamma_-=\{t\geqslant 0: \Delta A_t^{(2)}>0\},
\end{displaymath} then $\Gamma_+\cap\Gamma_-=\phi$. Rewrite the last
two expectations of (\ref{kmW}) as
\begin{eqnarray*}
&&E_{\bf x}\left(\sum_{t\in\Gamma_+}e^{-\alpha
t}\int_{X_{nt^-}}^{X_{nt^-}+\Delta A_t^{(1)}}\left[
\frac{\partial}{\partial x_n}W({\bf X}_t)+f_1({\bf
X}_t)\right]dy\right)\\
&&+E_{\bf x}\left(\sum_{t\in\Gamma_-}e^{-\alpha
t}\int_{X_{nt^-}-\Delta A_t^{(2)}}^{X_{nt^-}}\left[
-\frac{\partial}{\partial x_n}W({\bf X}_t)+f_2({\bf
X}_t)\right]dy\right).
\end{eqnarray*} By Theorem \ref{Wcond} this quantity is
nonnegative, and this shows $k_{\mathcal{S}}({\bf x})\geqslant
W({\bf x}),\forall {\bf x}\in S$.

Due to the extension (\ref{extcost}), we proved
$k_{\mathcal{S}}({\bf x})\geqslant W({\bf x}),\forall {\bf
x}\in{\mathbb{R}}^n$.

\item If
$S={\mathbb{R}}^{n-1}\times[\tilde{a},\tilde{b}]$ and the process
${\bf X}_t$ is the reflecting diffusion on $S$, then by Theorem
\ref{Wcond}, the first integral in (\ref{kmW}) is obviously zero. As
to the second and third integrals in (\ref{kmW}), because
$dA_t^{(1)},dA_t^{(2)}$ are zero whenever ${\bf X}_t$ is in
${\mathbb{R}}^{n-1}\times({a},{b})$, while at the boundary where
$A_t^{(1)},A_t^{(2)}$ increases, the integrands are zero, these two
integrals are zero too. The last two expectations are also zero due
to this construction hence $W({\bf x})= k_{\mathcal S}({\bf
x}),\forall {\bf x}\in S$.

On the other hand, suppose $W({\bf x})= k_{\mathcal S}({\bf
x}),\forall {\bf x}\in S$, then all the expectations in (\ref{kmW})
must be zero. Assume $S={\mathbb{R}}^{n-1}\times[\beta,\gamma]$ and
at least one of the inequalities is true: $\beta(\bar{\bf x})\neq
g_1(\bar{\bf x}),\gamma(\bar{\bf x})\neq g_2(\bar{\bf x})$, then due
to the continuity of these four functions we know that the sum of
the first three integrals in (\ref{kmW}) is positive by Theorem
\ref{Wcond}. And because the sum of the last two expectations in
(\ref{kmW}) is nonnegative, it can be seen that $W({\bf x})<
k_{\mathcal S}({\bf x})$. Therefore in order to have $W({\bf x})=
k_{\mathcal S}({\bf x})$, $S$ must be the region
${\mathbb{R}}^{n-1}\times[\tilde{a},\tilde{b}]$.

Again by Theorem \ref{Wcond}, we see that the processes ${\bf X}_t$
and $A_t^{(i)}$ ($i=1,2$) must all be continuous in order to
eliminate the last two expectations in (\ref{kmW}), which implies
$A_t^{(i)}=A_t^{(i)c}$ ($i=1,2$) when $\beta(\bar{\bf
x})=\tilde{a}(\bar{\bf x}),\gamma(\bar{\bf x})=\tilde{b}(\bar{\bf
x})$. Therefore $({\bf X}_t,A_t^{(1)},A_t^{(2)})$ must be the
reflecting diffusion on
${\mathbb{R}}^{n-1}\times[\tilde{a},\tilde{b}]$.
\end{enumerate}
\end{proof}
\begin{remark}
The possible jumps, $\Delta A_t^{(i)},i=1,2,$ only happen at time
zero. When the process ${\bf X}_t$ starts at a point outside the
region ${\mathbb{R}}^{n-1}\times[\tilde{a},\tilde{b}]$, the control
brings it back to this region immediately, and after that, the
process will be a continuous reflected diffusion. The confirmation
of the last assertion is shown below.
\end{remark}

If we let $\gamma=(0,0,...,0,1)^T$, then the reflected diffusion can
be written as
\begin{equation}\label{refdif}
d{\bf X}_t=\mu({\bf X}_t) dt+\sigma({\bf X}_t) d{\bf B}_t+\gamma
dA_t^{(1)}-\gamma dA_t^{(2)},\ \ t>0,
\end{equation} where $A_t^{(1)}$ increases only at the boundary $\tilde{a}(\cdot)$
and $A_t^{(2)}$ increases only at the boundary $\tilde{b}(\cdot)$.

We notice that the reflection only happens to the last component of
the process. Since the two curves $\tilde{a}(\cdot)$ and
$\tilde{b}(\cdot)$ are smooth and uniformly Lipschitz, if we let
$n({\bf x})$ be the inward normal for ${\bf x}$ at the boundary,
then we can show that there exist positive constants $\nu_1,\nu_2$
such that
\begin{eqnarray*}
\forall {\bf x}&=&(\bar{\bf x},\tilde{a}({\bar{\bf x}})),\
(\gamma,n({\bf
x}))\geqslant \nu_1,\\
\forall {\bf x}&=&(\bar{\bf x},\tilde{b}({\bar{\bf x}})),\
(\gamma,n({\bf x}))\leqslant -\nu_2.
\end{eqnarray*}
By a localization technique and Theorem 4.3 in \cite{Lions84}, it
can be shown that there exists a solution $({\bf
X}_t,A_t^{(1)},A_t^{(2)})$ to the reflected diffusion
(\ref{refdif}). This problem is called the Skorohod problem.

\section*{Concluding Remarks} In this paper, we studied a
multi-dimensional stochastic singular control problem via Dynkin
game and Dirichlet form. The value function of the Dynkin game
satisfies a variational inequality problem, and the integrated form
of this value function turns out to be the value function of the
singular control problem. By characterizing the regularities of the
value function of the Dynkin game and its integrated version, we
showed the existence of a classical solution to the
Hamilton-Jacobi-Bellman equation associated with this
multi-dimensional singular control problem, and this kind of
problems were traditionally solved through viscosity solutions. We
also proved that,  under some conditions, the optimal control policy
is given by two curves and the controlled process is the reflected
diffusion between these two curves. Unlike the one dimensional
singular control problem, where under some conditions, the boundary
of the optimal continuation region are given by two points
\cite{Fuku02}, it is much more difficult to characterize the
boundaries of the continuation region in the multi-dimensional
singular control problem. This paper investigates some conditions on
the regularity of value function and the form of optimal singular
control policies of multi-dimensional diffusion, and it provides a
basis for the search of further conditions and further regularities
in this realm.

\section*{Appendix} In this appendix we shall correct an error found in the paper
\cite{Fuku02}. In the paper ``Dynkin Games Via Dirichlet Forms and
Singular Control of One-Dimensional Diffusion''\cite{Fuku02},  the
authors tried to show the existences of a smooth value function and
an optimal policy to a one-dimensional stochastic singular control
problem via Dynkin game and Dirichlet form. The value function
$V(x)$ of a Dynkin game is known to exist \cite{Zab84}, which is the
solution of a variational inequality problem involving Dirichlet
form. The integration of $V(x)$ turns out to be a smooth optimal
return function $W(x)$ for a stochastic singular control problem.
Thus the traditional technique of viscosity solution is avoided.

In their paper, the underlying process is a generalized one
dimensional diffusion process given by
$dX_t=\mu(X_t)dt+\sigma(X_t)dw_t$, in which $w_t$ is a Wiener
process. It is found that a different diffusion process should be
considered in the proofs,  and as a result the main theorem of this
paper should be amended.

In their paper, the infinitesimal generator is defined as (see page
693, Eq. 4.1 in \cite{Fuku02})
\begin{equation}\label{gen}
Lu(x)=\frac{d}{dm}\frac{d}{ds}u(x)=\mu(x)u'(x)+\frac{1}{2}\sigma(x)^2
u''(x),
\end{equation}
where $ds(x)=\dot{s}(x)dx,dm(x)=\dot{m}(x)dx$, and (see Eq. 4.2 in
\cite{Fuku02})
\begin{equation}\label{sam}
\dot{s}(x)
=\exp\left(-\int_0^x\frac{2\mu(y)}{\sigma(y)^2}dy\right),\quad
\dot{m}(x) =
\frac{2}{\sigma(x)^2}\exp\left(\int_0^x\frac{2\mu(y)}{\sigma(y)^2}dy\right).
\end{equation}

The value function $W(x)$ of the stochastic singular control problem
is assumed to satisfy the following PDE (see Eq. 3.23 on page 693 in
\cite{Fuku02})
\begin{equation}
\alpha W(x) -\frac{d}{dm}\frac{d}{ds}W(x) = h(x),
\end{equation}
or equivalently
\begin{equation}
\alpha W(x) -\mu(x)W'(x)-\frac{1}{2}\sigma(x)^2 W''(x) = h(x),
\end{equation}
where (see Eq. 3.21 and Eq. 3.22 in \cite{Fuku02})
\begin{equation}\label{defh}
h(x) = \int_0^x H(y)\dot{s}(y)dy + C,
\end{equation}
and
\begin{equation}\label{defW}
W(x) = \int_a^x
V(y)\dot{s}(y)dy+\frac{1}{\alpha}\left(-\frac{f_1'(a)}{\dot{m}(a)}+h(a)\right).
\end{equation}

Then in the proof of Theorem 3.2 on page 693 in \cite{Fuku02}, the
authors constructed the function
\begin{equation}
U(x)=\alpha W(x)-\frac{d}{dm}\frac{d}{ds}W(x)-h(x),
\end{equation}
and claimed that
\begin{equation}
\frac{1}{\dot{s}(x)}U'(x) = \alpha
V(x)-\frac{d}{ds}\frac{d}{dm}V(x)-H(x).
\end{equation}
This is equivalent to
\begin{equation}\label{defU}
U(x)=\alpha W(x)-\mu(x)W'(x)-\frac{1}{2}\sigma(x)^2W''(x)-h(x),
\end{equation}
and
\begin{equation}\label{hjbV}
\frac{1}{\dot{s}(x)}U'(x) = \alpha
V(x)-\mu(x)V'(x)-\frac{1}{2}\sigma(x)^2V''(x)-H(x).
\end{equation}
However, by a careful examination, it can be seen that the above
proposition is not true in general. The reason here is that $\mu(x)$
and $\sigma(x)$ are both functions of $x$, and when taking the
derivative of $U(x)$, the product rule has to be applied. The
details are shown below.

By the definition of $W(x)$ in (\ref{defW}), $h(x)$ in (\ref{defh})
and $\dot{s}(x)$ in (\ref{sam}) in \cite{Fuku02}, we get
\begin{eqnarray}
W'(x) &&= V(x)\dot{s}(x),\nonumber\\
h'(x) &&= H(x)\dot{s}(x),\nonumber\\
\ddot{s}(x)&&=-\dot{s}(x)\frac{2\mu(x)}{\sigma(x)^2},\nonumber
\end{eqnarray} hence
\begin{eqnarray}
W''(x) &&=
V'(x)\dot{s}(x)-V(x)\dot{s}(x)\frac{2\mu(x)}{\sigma(x)^2},\nonumber\\
W'''(x)&&=V''(x)\dot{s}(x)-2V'(x)\dot{s}(x)\frac{2\mu(x)}{\sigma(x)^2}+V(x)\dot{s}\frac{4\mu(x)^2}{\sigma(x)^4}\nonumber\\
&&\quad
-V(x)\dot{s}(x)\left(2\mu'(x)\sigma(x)^{-2}-4\mu(x)\sigma(x)^{-3}\sigma'(x)\right).\nonumber
\end{eqnarray}
Now if we take the derivative of $U(x)$ in (\ref{defU}) we get
\begin{eqnarray}
U'(x)=&&\alpha
V(x)\dot{s}(x)-\mu'(x)V(x)\dot{s}(x)-\mu(x)V'(x)\dot{s}(x)+\mu(x)V(x)\dot{s}(x)\frac{2\mu(x)}{\sigma(x)^2}\nonumber\\
&&-\sigma(x)\sigma'(x)V'(x)\dot{s}(x)+\sigma(x)\sigma'(x)V(x)\dot{s}(x)\frac{2\mu(x)}{\sigma(x)^2}\nonumber\\
&&-\frac{1}{2}\sigma(x)^2\left(V''(x)\dot{s}(x)-2V'(x)\dot{s}(x)\frac{2\mu(x)}{\sigma(x)^2}+V(x)\dot{s}\frac{4\mu(x)^2}{\sigma(x)^4}\right.\nonumber\\
&&\quad
\left.-V(x)\dot{s}(x)\left(2\mu'(x)\sigma(x)^{-2}-4\mu(x)\sigma(x)^{-3}\sigma'(x)\right)\right)-H(x)\dot{s}(x).\nonumber
\end{eqnarray}
After simplifying this expression and comparing it with (\ref{hjbV})
we should have the following
\begin{equation}
0=-\sigma(x)\sigma'(x)V'(x)+2\mu(x)V'(x),
\end{equation}
which does not hold in general. The following condition should be
added to make it hold.
\begin{equation}\label{result}2\mu(x)=\sigma(x)\sigma'(x)\end{equation}

A second concern of this paper might be more profound. The Dirichlet
form in this paper is defined as (see Eq. 3.3 on page 686 in
\cite{Fuku02})
\begin{equation}\label{dform}
\mathcal{E}(u,v)=\int_{-A}^Au'(x)v'(x)\frac{1}{\dot{m}(x)}dx,\quad
u,v\in\mathcal{F},
\end{equation}
where
\begin{eqnarray}
\mathcal{F}&&=H^1((-A,A);dx)\nonumber\\
&&=\{u\in L^2((-A,A);dx):\ u{\rm\ is\ absolutely\ continuous,\
}u'\in L^2((-A,A);dx)\}.\nonumber
\end{eqnarray}

The authors claimed that this Dirichlet form
$(\mathcal{E},\mathcal{F})$ is regular on $L^2([-A,A];ds)$ and the
associated underlying process is a reflecting barrier diffusion on
$[-A,A]$ with infinitesimal generator $\frac{d}{ds}\frac{d}{dm}$,
i.e., the generator $L$ given in (\ref{gen}). The correspondence is
given by (see Corollary 1.3.1 on page 21 of \cite{Fuku11})
\begin{equation}\label{da}
\mathcal{E}(u,v)=(-Lu,v),\quad u\in \mathcal{D}(L),v\in\mathcal{F},
\end{equation}
where $\mathcal{D}(L)$ is the domain of $L$. Since the underlying
process is a reflecting barrier diffusion on $[-A,A]$,
$\mathcal{D}(L)$ is given by (see page 22 of \cite{Fuku11})
\begin{eqnarray}
\mathcal{D}(L)=\{&&u\in\mathcal{F}: u'{\rm\ is\ absolutely\
continuous,\ }\nonumber\\
&&u''\in L^2((-A,A);dx), u'(-A)=u'(A)=0\}.\nonumber
\end{eqnarray}

Now we try the integration by parts on (\ref{dform}) and get
\begin{eqnarray}
\mathcal{E}(u,v)&&=\int_{-A}^A
u'(x)v'(x)\frac{\sigma(x)^2}{2}\exp\left(-\int_0^x\frac{2\mu(y)}{\sigma(y)^2}dy\right)dx\nonumber\\
&&=-\int_{-A}^A\left(\frac{\sigma(x)^2}{2}u''(x)+\sigma(x)\sigma'(x)u'(x)-\mu(x)u'(x)\right)v(x)\exp\left(-\int_0^x\frac{2\mu(y)}{\sigma(y)^2}dy\right)dx.\nonumber
\end{eqnarray}

Once again, when the condition (\ref{result}) holds, we get
\begin{displaymath}
\frac{\sigma(x)^2}{2}u''(x)+\sigma(x)\sigma'(x)u'(x)-\mu(x)u'(x)=\frac{\sigma(x)^2}{2}u''(x)+\mu(x)u'(x)=Lu(x),
\end{displaymath} and (\ref{da}) holds.

As a conclusion, if the condition (\ref{result}) is added, then all
the results in that paper still hold, but for a very particular Ito
diffusion.

In the following we give another way to amend the results of that
paper which makes the theorems more general. If we just simply
consider the diffusion
\begin{equation}\label{ndiff}
dX_t = \gamma(X_t)dt + \sigma(X_t)dw_t,
\end{equation} where
\begin{displaymath}
\gamma(x)=\sigma(x)\sigma'(x)-\mu(x),
\end{displaymath}
and define the infinitesimal generator
\begin{displaymath}
L_\gamma u(x) = \gamma(x)u'(x)+\frac{1}{2}\sigma(x)^2 u''(x),
\end{displaymath}
while the Dirichlet form is still defined as in (\ref{dform}) and
$\dot{s}(x),\dot{m}(x)$ are still given in (\ref{sam}), then we get
\begin{displaymath}
\mathcal{E}(u,v)=(-L_\gamma u,v).
\end{displaymath} That means the underlying process associated with
the Dirichlet form (\ref{dform}) should be (\ref{ndiff}). With this
in mind, we can examine again the results of that paper
\cite{Fuku02}. Results in Section 2 are classical on variational
inequalities and optimal stopping. In Section 3, the part
$\frac{d}{ds}\frac{d}{dm}$, whenever it appears before Theorem 3.2,
should be replace by $L_\gamma$. Let the functions $h(x),W(x)$ still
be defined as in Eqn. (3.21) (3.22) in that paper, respectively,
then Theorem 3.2 holds intact. But in the proof of this theorem,
after setting
\begin{eqnarray}
U(x)&=&\alpha W(x)-\frac{d}{dm}\frac{d}{ds}W(x)-h(x)\nonumber\\
&=&\alpha
W(x)-\mu(x)W'(x)-\frac{1}{2}\sigma(x)^2W''(x)-h(x)\nonumber,
\end{eqnarray} and taking the derivative of both sides, we should
get
\begin{eqnarray}
\frac{1}{\dot{s}(x)}U'(x) &=& \alpha
V(x)-(\sigma(x)\sigma(x)'-\mu(x))V'(x)-\frac{1}{2}\sigma(x)^2V''(x)-H(x)\nonumber\\
&=&\alpha V(x)-L_\gamma V(x)-H(x).\nonumber
\end{eqnarray} Since Theorem 3.1 has been amended, the rest of the
proof of Theorem 3.2 just follows. Section 4 of that paper is about
a verification theorem, and the results there still hold.

It might be interesting to notice that when $\sigma$ is a constant,
we get $\gamma(x)=-\mu(x)$, and
\begin{equation}
\mathcal{E}(u,v)=\int_{-A}^A
u'(x)v'(x)\frac{\sigma^2}{2}\exp\left(\int_0^x\frac{2\gamma(y)}{\sigma^2}dy\right)dx.
\end{equation}

\section*{Acknowledgments}
In memory of Dr. Michael Taksar.


\end{document}